\documentclass[12pt]{article}
\usepackage{amssymb,amsmath}
\usepackage[english]{babel}
\usepackage{amsfonts}
\usepackage{dsfont}
\usepackage{amsmath,amssymb,amsthm}
\usepackage{eucal}
\usepackage{latexsym}
\usepackage{amsbsy}
\usepackage{amsmath}
\usepackage{amscd}
\usepackage{amsgen}
\usepackage{amsopn}
\usepackage{amstext}
\usepackage{ifthen}
\usepackage{amsxtra}
\usepackage{amsfonts}
\usepackage{float}
\usepackage{xcolor}
\def \R{{\hbox{\vrule width 0.6pt height 6.8pt depth -.2pt\kern-0.2pt
R}}}
\def \P{{\hbox{\vrule width 0.6pt height 6.8pt depth -.2pt\kern-0.2pt
P}}}

\usepackage[english]{babel}
\usepackage[latin1]{inputenc}
\usepackage{latexsym}

\def \R {\mathbb R}
\def \N {\mathbb N}
\def \P {\mathbb P}

\def\MOD#1{{|\kern -.16em |\kern -.16em | #1 | \kern -.16em |\kern
 -.16em |}}
\def \epsilon {\varepsilon}

\newtheorem{theo}{\bf THEOREM}[section]
\newtheorem{lem}[theo]{\bf LEMMA}
\newtheorem{pro}[theo]{\bf PROPOSITION}
\newtheorem{cor}[theo]{\bf COROLLARY}
\newtheorem{defi}[theo]{\bf DEFINITION}
\newtheorem{rem}[theo]{\bf REMARK}
\newcommand{\eps}{\varepsilon}
\newenvironment{proof_prop3.5}[1][\bf{Proof of Proposition \ref{prop-reg}}]{\noindent{\it{#1}}}{\hfill$\square$\\ }
\newenvironment{proof_prop3.9}[1][\bf{Proof of Proposition 3.10}]{\noindent{\it{#1}}}{\hfill$\square$\\ }
\newenvironment{proof_prop3.8}[1][\bf{Proof of Proposition 3.8}]{\noindent{\it{#1}}}{\hfill$\square$\\ }
\newenvironment{proof_prop3.3}[1][\bf{Proof of Proposition 3.4}]{\noindent{\it{#1}}}{\hfill$\square$\\ }
\newenvironment{proof_prop3.4}[1][\bf{Proof of Proposition
  3.5}]{\noindent{\it{#1}}}{\hfill$\square$\\ }

\numberwithin{equation}{section}
\begin{document}
\begin{center}
{\Large{\sc Construction of a blow-up solution for a perturbed nonlinear heat equation with a gradient  and a non-local  term  }}

\bigskip

Bouthaina Abdelhedi

\medskip
\centerline{\it \small Department of Mathematics}
\centerline{\it\small Faculty of Sciences of Sfax,}
\centerline{\it\small  BP1171, Sfax 3000, Tunisia}
 bouthaina.abdelhedi@fss.usf.tn\\

  \bigskip
  Hatem Zaag
    \medskip
  \centerline{\it \small CNRS LAGA (UMR 7539)}\\
  \centerline{\it\small Universit\'e Sorbonne Paris Nord,}
  \centerline{\it\small  99 Avenue Jean-Baptiste Cl\'ement 93430 Villetaneuse, France}
   hatem.zaag@univ-paris13.fr
\end{center}

\begin{abstract} We consider in this paper a perturbation of the standard semilinear heat equation by a term involving the space derivative and a non-local term. We prove the existence of a blow-up solution, and give its blow-up profile.
Our proof relies on the  following   method: we  linearize the equation (in similarity variables) around the expected profile, then, we control the nonpositive directions of the spectrum thanks to the decreasing properties of the kernel. Finally,  we use a topological argument to control the positive directions of the spectrum.
\end{abstract}
~~
\section{Introduction}
We are interested in this paper in  the following nonlinear parabolic equation
\begin{equation}\label{eq_u}
\left \{
\begin{array}{lcl}
 u_{t}&=&\Delta u+|u|^{p-1}u+\mu |\nabla u|\displaystyle \int_{B(0, |x|)}|u|^{q-1},\\
 u(0)&=&u_0\in W^{1, \infty}(\R^N),
\end{array}
\right.
\end{equation}
where $u=u(x,t)\in \R$, $x\in \R^N$ and the parameters $p, q$ and $\mu$ are such that \begin{equation}\label{hyp}
\displaystyle  p>3, \quad \frac{N}{2}(p-1)+1<q<\frac{N}{2}(p-1) +\frac{p+1}{2},\; \mu\in \R.
\end{equation}
When $\mu=0$,  blow-up results for the equation $(\ref{eq_u})$  have been extensively studied.\\
The existence of blow-up solutions has been proved by several authors, see Fujita \cite{fujita}, Ball \cite{ball}.
 We say that  $u$ blows up in finite time $T$ in the sense that $$\|u(t)\|_{L^\infty}\to \infty, \quad t\to T.$$
We call $T$ the blow-up time of $u$.\\
Many works have been describing the asymptotic blow-up behavior near a given blow-up point, see Giga and Kohn \cite{giga1}, \cite{giga2}, Weissler \cite{weissler}, Filippas, Kohn and Liu \cite{filippas1}, \cite{filippas2}, Herrero and Vel\'azquez \cite{herrero1}, \cite{herrero2}, \cite{herrero3},
\cite{herrero4},  Merle and Zaag \cite{MZ97}, \cite{MZ98}, \cite{MZ}.\\
Also, many papers have been  devoted to the blow-up profile; see Bricmont and Kupiainen \cite{bricmont}, Merle and Zaag \cite{MZ97}, Berger and Kohn \cite{berger} and Nguyen and Zaag \cite{NZ1}, \cite{NZ2}.\\
Particulary, these authors  constructed a solution $u$ which approaches an explicit universal profile $f$ depending only on $p$ and independent from initial data  as follows

\begin{equation}
\|(T-t)^{\frac{1}{p-1}} u(x,t)-f(\frac{x}{\sqrt{(T-t)|\log(T-t)|}})\|_{L^\infty}\to_{t\to T}0,
\end{equation}
where $f$ is the profile defined by
\begin{equation}\label{profile}
\displaystyle f(z)=(p-1+\frac{(p-1)^2}{4p}|z|^2)^{-\frac{1}{p-1}}.
\end{equation}
Such a construction relies on a two-step method:
\begin{itemize}
\item
The guess of the limiting profile, based on a formal approach in the so-called similarity variables (defined in $(\ref{var_sim})$ below); this is particularly well explained in Berger and Kohn \cite{berger}, Filippas and  Kohn  \cite{filippas1} and  Bricmont and Kupiainen \cite{bricmont}.
\item The rigorous proof performed in similarity variables, where
the authors linearize the equation around the introduced profile, and control the nonpositive part of the spectrum thanks to the decaying properties of the linear operator. Then, they use a topological argument for  the  positive directions of the spectrum.
\end{itemize}
An interesting  question following the above results is to tell how robust is the construction method?\\
A first result in that direction was obtained for the following
 equation with a gradient term:
\begin{equation}\label{1.4.5}
u_{t}=\Delta u+|u|^{p-1}u+\mu |\nabla u|^q.
\end{equation}
For this equation, we mention
the  blow-up results   obtained by  Souplet, Tayachi and Weissler \cite{STW}, when $q=\frac{2p}{p+1}$, $\mu<0$,  Galaktionov and Vazquez \cite{galak1}, \cite{galak2}, when $q=2$ and $\mu>0$,  Ebde and Zaag \cite{EZ}, when $q<\frac{2p}{p+1}$ and  Tayachi and Zaag \cite{TZ}, when $q=\frac{2p}{p+1}$, $\mu>0$. There is also a numerical result  by Nguyen \cite{N}. In \cite{EZ} and \cite{TZ}, the authors derive the blow-up profile.
Because of the presence of the perturbation involving a nonlinear gradient term, they obtain the convergence in $W^{1, \infty}(\R^N)$.\\
We would like to mention that the construction method has proved to be successful in a different class of PDEs involving non-local terms, namely the following equation modeling Micro Electrical Mechanical Systems (MEMS):
$$u_{t}=\Delta u+\displaystyle \frac{\lambda}{(1-u)^2(1+\gamma\int_\Omega\frac{1}{1-u}dx)^2},$$
see Duong and Zaag \cite{DZ}.\\
In this paper, we would like to consider a mixed-type equation involving a gradient term together with  a non-local term, namely, equation $(\ref{eq_u})$.
 We note that the equation $(\ref{eq_u})$ is a new  class of perturbed semilinear  heat equation, but compared to the previous works, our perturbation is not trivial since we have both a non-local and a  gradient term. \\

The aim of this paper is to  construct a  solution of the equation $(\ref{eq_u})$ which approaches the same profile $f$ as for the case $\mu=0$. More precisely, we prove the following result.
\begin{theo}\label{th1}~\\
Let $\mu\in \R$, $p>3,$ and $q\in \R$ such that $\displaystyle
\frac{N}{2}(p-1)+1<q<\frac{N}{2}(p-1) +\frac{p+1}{2}$. Consider then an arbitrary   $\beta$  such that
\begin{equation}\label{beta}
 0<\beta <\frac{2}{p-1},\;\mbox{if}\;\mu=0\; \mbox{ and }\;\displaystyle \frac{N}{q-1}<\beta<\frac{2}{p-1},\;\mbox{if}\; \mu\neq 0.
\end{equation}
There exists $T>0$ such that equation $(\ref{eq_u})$ has a solution $u(x,t)$ such that $u$  blows up at time  $T$ at the point $a=0$. Moreover,  for all $t\in [0, T)$, for all $x\in \R^N$,
$$\displaystyle |u(x,t)-(T-t)^{-\frac{1}{p-1}}f(\frac{x}{\sqrt{(T-t)|\log(T-t)|}})|\leq \frac{C}{1+(\frac{|x|^2}{T-t})^{\frac{\beta}{2}}}\frac{(T-t)^{-\frac{1}{p-1}}}{|\log(T-t)|^{\frac{1-\beta}{2}}},$$
and
$$\displaystyle |\nabla u(x,t)-\frac{(T-t)^{-\frac{1}{2}-\frac{1}{p-1}}}{\sqrt{|\log(T-t)|}}\nabla f(\frac{x}{\sqrt{(T-t)|\log(T-t)|}})|\leq \frac{C}{1+(\frac{|x|^2}{T-t})^{\frac{\beta}{2}}}\frac{(T-t)^{-\frac{1}{2}-\frac{1}{p-1}}}{|\log(T-t)|^{\frac{1-\beta}{2}}},$$
where $f(z)=(p-1+b|z|^2)^{-\frac{1}{p-1}}$, $z\in \R^N$, $b=\frac{(p-1)^2}{4p}$.\\
\end{theo}
\begin{rem}When $\mu \neq 0$, we see from  \eqref{hyp} that $\frac N{q-1}<\frac 2{p-1}$, hence, condition \eqref{beta} is not empty.
\end{rem}
\begin{rem}
Note that our Theorem  improves  the results  of \cite{MZ97} already  in the case of the standard nonlinear heat equation, i.e. when $\mu=0$. Indeed,  our profile  is sharper that the profile derived in \cite{MZ97}, in the sense that we divide here the bound by $1+(\frac{|x|^2}{(T-t)})^{\frac{\beta}{2}}$. Thus,  we find more information about the tail (in space) of the blow-up solution.
\end{rem}
More precisely, the solution constructed in Theorem $\ref{th1}$ satisfies the following result:
\begin{cor}\label{corol}
Let $u$ be the solution of $(\ref{eq_u})$ constructed in Theorem  $\ref{th1}$ and $T$ its blow-up time.\\
For  all  $|x|\ge (T-t)^\frac{1}{2} |\log(T-t)|^{\frac{1}{2}[1+(\frac{2}{p-1}-\beta)^{-1}]}$, we have 
 $$\displaystyle |u(x,t)|\leq \frac{C}{|x|^\beta}\frac{(T-t)^{\frac{\beta}{2}-\frac{1}{p-1}}}{| \log(T-t)|^\frac{1-\beta}{2}}\leq C\frac{(T-t)^{-\frac{1}{p-1}}}{| \log(T-t)|^{\frac{1}{2-\beta(p-1)}}},$$
and
$$\displaystyle |\nabla u(x,t)|\leq \frac{C}{|x|^\beta}\frac{(T-t)^{\frac{\beta}{2}-\frac{1}{p-1}-\frac{1}{2}}}{|\log(T-t)|^\frac{1-\beta}{2}}\leq C\frac{(T-t)^{-\frac{1}{p-1}-\frac{1}{2}}}{| \log(T-t)|^{\frac{1}{2-\beta(p-1)}}}.$$
\end{cor}

\begin{rem}\label{rem}
We suspect the origin to be the only blow-up point of $u$. Unfortunately, because of the non-local term in equation $(\ref{eq_u})$, we couldn't apply the localization and
iteration method presented by Giga and Kohn in Theorem 2.1 page $850$ of \cite{giga1}, to prove the single-point blow-up property.\\
Nevertheless, we see from Corollary \ref{corol} that for any $x_0\in \R^N\backslash\{0\}$ and in some cylinder around $(x_0, T)$ the solution is uniformly negligible with respect to the ODE rate $(T-t)^{-\frac{1}{p-1}}$, which is in our opinion a strong evidence  showing that the solution doesn't blow up at $x_0$.  More precisely,  for any $x_0\in \R^N \backslash\{0\}$ and $\eps_0>0$, there exists $t_0(x_0, \eps_0)<T$ such that for all $t\in [t_0(x_0, \eps_0),T)$ and  $x\in B(x_0, \frac{|x_0|}{2})$,
$$|u(x,t)|<\frac{\eps_0}{(T-t)^{\frac{1}{p-1}}}.$$
\end{rem}
\begin{rem}
The local Cauchy problem for  equation (\ref{eq_u}) can be solved in the weighted  functional space  \begin{equation}\label{fun_spa}
W^{1, \infty}_\beta(\R^N)=\{g; \; (1+|y|^{\beta}) g \in L^{\infty},\; (1+|y|^{\beta})\nabla g \in L^{\infty } \},
\end{equation} using a fixed point argument. For the reader's convenience we  prove this  result in Appendix $C$.
\end{rem}

Thanks to a better choice of initial data, we may derive this stronger version of Theorem \ref{th1}:
\begin{cor}\label{cor_syme}
\begin{enumerate}
\item By a good choice of parameters, we can make sure that the solution constructed in Theorem \ref{th1} is nonnegative and symmetric with respect to the origin.
\item If in addition $\mu<0$, then the solution $u$  constructed  in Theorem \ref{th1} blows up only at the origin.
\end{enumerate}
\end{cor}
\begin{rem} This statement can be easily derived from the proof of Theorem \ref{th1}. For that reason, we only explain the argument here and won't give the proof.
 Regarding part 1, since equation \eqref{eq_u} preserves radial symmetry and nonnegativity, one has simply to impose $d_1=0$ in the expression of initial data given in \eqref{ci} below, and all the proof runs as we wrote it.\\ 
As for part 2, simply note that the sign of $\mu$ allows us to derive an upper bound, thanks to the framework given by Giga and Kohn in Section 2 page 850 of \cite{giga3} (see also page 245 of the book of Quittner and Souplet \cite{QS} for another version of the proof). Since the solution is nonnegative by construction, we conclude the proof.
\end{rem}
\begin{rem} One may wonder why we keep the 2 statements, Theorem \ref{th1} and Corollary \ref{cor_syme}, as Corollary \ref{cor_syme} seems to be stronger. In fact, we keep Theorem \ref{th1}, since we believe it will be  crucial in deriving the stability of the behavior described in Theorem \ref{th1}, with respect to initial data, not only in the radially symmetric class.  Indeed,  as one may see from the case of the standard semilinear heat equation in Merle and Zaag \cite{MZ97} or equation  \eqref{1.4.5} in Tayachi and Zaag \cite{TZ},   stability is a natural by-product of the existence proof, thanks to a geometrical interpretation of the parameters of the finite dimensional problem (i.e $d_0$, $d_1$ in \eqref{ci}) in terms of the blow-up time and the blow-up point.\\
Now, because of the non-local term, equation \eqref{eq_u} is not invariant with respect to translation in time, and this makes a serious difficulty in deriving stability from the existence proof. For that reason, we leave stability to a future work. However, we keep the statement of Theorem \eqref{th1}, since we believe it will be an important piece in the stability proof. 
\end{rem}
The proof of Theorem \ref{th1} is  based on techniques developed by Bricmont and Kupiainen \cite{bricmont}, Merle and Zaag \cite{MZ97} and Tayachi and Zaag \cite{TZ}. This is reasonable since in similarity variables defined below by $(\ref{var_sim})$, the new perturbation term comes with an exponentially decreasing coefficient. Although these modifications do not affect the general framework developed  in the previous work, we need to perform   some crucial modifications with respect to  results found in \cite{bricmont}, \cite{MZ97}, \cite{EZ},     \cite{TZ} in order to control the new term.  Let us mention the crucial modifications:
\begin{itemize}
\item We modify the functional space. Since the perturbation contains $\displaystyle \int_{B(0, |x|)}|u|^{q-1}$, our proofs need some involved argument to control this term. In particular, we need to study the convergence in the new functional space  $W^{1, \infty}_\beta(\R^N)$ defined in \eqref{fun_spa}.
More specifically, some involved parabolic regularity argument is proved to handle the gradient term.\\
\item In order to study  blow-up in the new functional space, we need to modify the definition of the shrinking set (see Definition \ref{def_shrinking} below).  Therefore,   some crucial estimates are needed.\\
\item  Finally, we linearize the equation around a new profile given by $(\ref{n_profile})$ below, showing a cut-off at infinity, so that the profile belongs to the space defined in \eqref{fun_spa}. A good  understanding of the linearized operator is necessary in order to handle the new shrinking set.\\
\end{itemize}

Let us remark that the construction method involves the linearization of the equation (in similarity variables defined below in $(\ref{var_sim})$), with a different treatment of the projections on the solution,  according to the sign of the eigenvalues in the specified eigenspace:
\begin{itemize}
\item The infinite dimensional component of the solution, corresponding to the nonpositive part of the spectrum, is controlled thanks to the decaying properties of the linearized operator; since the  positive part of the spectrum is finite dimensional, we call this step a \textit{finite dimensional reduction}.
\item Then, the positive part of the spectrum is controlled thanks to a topological argument, based on index theory.
\end{itemize}
\bigskip
This  paper is organized as follows. In Section $2$, we give a formulation of the problem. In Section $3$, we prove  the existence of a solution of equation $(\ref{eq_v})$. Finally, in Section $4$ we prove Theorem \ref{th1} and Corollary \ref{corol}. As we stated in Remark $1.8$ we don't prove Corollary \ref{cor_syme}. In the appendix, we prove various technical estimates and also the Cauchy problem for equation \eqref{eq_u}.
\section{Formulation of the problem}
A fundamental tool for the study of the asymptotic behavior of blow-up solutions is the following similarity variables framework  introduced by Giga and Kohn \cite{giga1}, \cite{giga2}, \cite{giga3}:
\begin{equation}\label{var_sim}
\displaystyle  y=\frac{x}{\sqrt{T-t}}, \;\;s=-\log(T-t)\;\;
\mbox{and}\;\;
 w(y,s)=\displaystyle (T-t)^{\frac{1}{p-1}}u(x,t),\end{equation}
where $T$ is the time where we want the solution to blow up.\\
Therefore, if $u(x,t)$ satisfies $(\ref{eq_u})$ for all $(x,t)\in \R^N\times [0, T)$, then $w(y,s)$ satisfies the following equation for all $(y,s) \in \R^N\times[-\log T, +\infty)$ :
\begin{equation}\label{eq_w}
w_{s}=\Delta w-\frac12 y.\nabla w-\frac{1}{p-1}w+|w|^{p-1}w+\mu e^{-\gamma s} |\nabla w|\int_{B(0, |y|)}|w|^{q-1},
\end{equation}
where $\gamma=\displaystyle \frac{p-q}{p-1}+\frac{N-1}{2}$.
\begin{rem}
 We would like to emphasize the fact that $\gamma>0$ from condition \eqref{hyp}, which explains the little effect of the perturbation  term for large time.
\end{rem}
The study of $u$ as $t\to T$ is equivalent to the study of the asymptotic behavior of $w$ as $s\to +\infty$.\\
We would like to find $s_0>0$ and  initial data  $w_0$ such that the solution $w$ of equation (\ref{eq_w}) with  $w(s_0)=w_0$ satisfies
$$\|w(y,s)-f(\frac{y}{\sqrt{s}})\|_{W^{1, \infty}_\beta}\to_{s\to \infty}0, $$
where $f$ is the profile defined by
\begin{equation}\label{profile}
\displaystyle f(z)=(p-1+\frac{(p-1)^2}{4p}|z|^2)^{-\frac{1}{p-1}}.
\end{equation}
In order to prove this, we will not linearize equation $(\ref{eq_w})$ around $f+\displaystyle\frac{\kappa N}{2ps}$ as in \cite{TZ} and  \cite{MZ97}, since this function is not in the space $W^{1, \infty}_\beta$. We will in fact linearize it   around
\begin{equation}\label{n_profile}
\displaystyle\varphi(y,s)=f(\frac{y}{\sqrt{s}})+\frac{\kappa N}{2ps}\chi_0(
\frac{y}{g_\eps(s)}),
\end{equation} where $\kappa=(p-1)^{-\frac{1}{p-1}}$ is a stationary solution for equation $(\ref{eq_w})$, $\chi_0\in C_0^\infty$ with $supp(\chi_0)\subset B(0, 2)$ and $\chi_0\equiv 1$ on $B(0, 1)$ and $g_\eps(s) =s^{\frac{1}{2}+\eps}$, where $\eps$ is a fixed constant satisfying $0<\eps<min(1, \frac{p-1}{4})$, (we could have taken $\eps=\frac{1}{2}\min(1, \frac{p-1}{4}))$.\\
We introduce now
\begin{equation}\label{def_v}
v(y,s)=w(y,s)-\varphi(y,s).
\end{equation}
If $w$ satisfies  equation $(\ref{eq_w})$ then $v$ satisfies the following equation
\begin{equation}\label{eq_v}
v_s=(\mathcal{L}+V)v+B(v)+R(y,s)+N(y, s),
\end{equation}
where
\begin{itemize}
\item the linear term is $(\mathcal{L}+V)v$ where 
\begin{equation}\label{linear}\mathcal{L}(v)=\Delta v -\frac12 y.\nabla v+v \; \mbox{and}\;V(y, s)=p\varphi^{p-1}-\frac{p}{p-1},
\end{equation}
\item the nonlinear term is
\begin{equation}\label{nonlinear-term}
B(v)=|v+\varphi|^{p-1}(v+\varphi)-\varphi^{p}-p\varphi^{p-1}v,
\end{equation} 
\item the rest term  involving $\varphi$ is
$$R(y,s)=\Delta \varphi -\frac12 y.\nabla \varphi -\frac{1}{p-1}\varphi+\varphi^p-\varphi_s,$$
\item and the new term is
\begin{equation}\label{new-term}
\displaystyle N(y,s)=\mu e^{-\gamma s}|\nabla v+\nabla \varphi|\int_{B(0, |y|)} |v+\varphi|^{q-1}.
\end{equation}
\end{itemize}
In comparison with the case of the equation without gradient  ($\mu=0$), all the terms in $(\ref{eq_v})$  were already present in \cite{MZ97}, \cite{TZ} and \cite{bricmont}, except the new term $N(y,s)$ which needs to be carefully studied.\\
In the following analysis, we will use the following integral form of the  equation $(\ref{eq_v})$. Let $K$ be  the fundamental solution of the operator $\mathcal{L}+V$. Then, for each $s\geq \sigma\geq s_0$,  we have
\begin{equation}\label{eq_int_v}
v(s)=K(s, \sigma) v(\sigma)+\int_\sigma^s K(s,t)(B(v(t))+R(t)+N(t))dt.
\end{equation}
Since the linear operator  $\mathcal{L}+V$ will play an important role in our analysis, we first need to  recall some  of its properties (for more details, see \cite{bricmont}).\\ The operator $\mathcal{L}$ is self-adjoint in $D(\mathcal{L})\subset L^2_\rho(\R^N)$, where $$L^2_\rho(\R^N)=\{v\in L^2_{loc}(\R^N);\quad \int_{\R^N}(v(y))^2 \rho(y) dy<\infty\}, \quad \rho(y)=\displaystyle \frac{e^{-\frac{|y|^2}{4}}}{(4\pi)^{\frac{N}{2}}}.$$
The spectrum of $\mathcal{L}$ consists only in eigenvalues  given by
\begin{equation}\label{spect}
 spec(\mathcal{L})=\{1-\frac{m}{2}; \quad m\in \N\}.
 \end{equation}
The eigenfunctions of $\mathcal{L}$ are derived from Hermite polynomials.\\
For $N=1$, all the eigenvalues are simple, and the eigenfunction corresponding to $1-\frac{m}{2}$ is
\begin{equation}\label{vect_prop}
\displaystyle h_m(y)=\sum_{k=0}^{[\frac{m}{2}]}\frac{m!}{k! (m-2k)!}(-1)^ky^{m-2k}.
\end{equation}
In particular $h_0(y)=1$, $h_1(y)=y$ and $h_2(y)=y^2-2$. Note that $h_m$ satisfies
$$\displaystyle \int_{\R} h_n h_m \rho dx=2^nn! \delta_{n,m}.$$
We also introduce $\displaystyle k_m=\frac{h_m}{\|h_m\|_{L^2_\rho(\R)}^2}$.\\
For $N\geq 2$, the eigenspace corresponding to $1-\frac{m}{2}$ is given by
$$E_m=\{h_{m_1}(y_1)\cdots h_{m_N}(y_N); \quad \quad m_1+\cdots m_N=m\}.$$
In particular,$$E_0=\{1\}, \; E_1=\{y_i;\quad  \; i=1\cdots N\}\; \mbox{ and} \;E_2=\{h_2(y_i),\;  y_iy_j;\quad   i,j=1, \cdots , N,\;  i\neq j\}.$$
The potential $V(y,s)$ has two fundamental properties:
\begin{itemize}
\item $V(., s)\to 0$ in $L^2_\rho$ as $s\to +\infty$. In particular the effect of $V$ on the bounded sets or in the "blow-up"  area $(|y|\leq K_0\sqrt{s})$ is regarded as a perturbation of the effect of $\mathcal{L}$.
\item Outside of the  "blow-up"  area, we have the following property: for all $\eps>0$ the exist $C_\eps>0$ and $s_\eps$ such that
$$\displaystyle \sup_{s\geq s_\eps , |y|\geq C_\eps \sqrt{s}}|V(y,s)-(-\frac{p}{p-1})|\leq \eps.$$
\end{itemize}
This means that $\mathcal{L}+V$ behaves like $\mathcal{L}-\frac{p}{p-1}$ in the region $|y|\geq K_0\sqrt{s}$. Because $1$ is the largest eigenvalue of $\mathcal{L}$, the operator $\mathcal{L}-\frac{p}{p-1}$ has a purely negative spectrum, which simplifies greatly the analysis in the outside of the  "blow-up"  area.\\
Since the behavior of $V$ inside and outside the  "blow-up"  area  different, we  decompose $v$ as follows.  We introduce the following cut-off function:
\begin{equation}\label{chi}
\chi(y, s) =\chi_0(\displaystyle \frac{|y|}{K_0\sqrt{s}}),
\end{equation}  where $K_0>0$ is chosen large enough so that various technical estimates hold, and the cut-off function $\chi_0$ was already introduced after \eqref{n_profile}.\\
We write $$v(y, s)= v_b(y,s)+v_e(y,s),$$
with $$v_b(y,s)=v(y,s)\chi(y,s), \quad v_e(y,s)=v(y,s)(1-\chi(y,s)).$$
We note that $supp v_b(s)\subset B(0, 2K_0\sqrt{s})$ and $supp v_e(s)\subset \R^N\backslash  B(0, K_0\sqrt{s})$.\\
In order to control $v_b$, we decompose it according  to the sign of the eigenvalues of $\mathcal{L}$ as follows (for simplicity, we give the decomposition only for $N=1$, bearing in mind that the situation for $N\ge 2$ is the same, except for some more complicated notations that can be found in Nguyen and Zaag \cite{NZ0}, where the formalism in higher dimensions is extensively given):
\begin{equation}\label{decomp}
\displaystyle v(y, s)= v_b(y,s)+v_e(y,s) =\sum_{m=0}^{2}v_m(s)h_m(y)+v_{-}(y, s)+v_e(y,s),
\end{equation}
where for $0\leq m\leq 2$, $v_m=P_m(v_b)$ and $v_{-}(s)=P_{-}(v_b)$, with $P_m$ being the $L^2_\rho$ projector on $h_m$, the eigenfunction corresponding to $\lambda=1-\frac{m}{2}$, and $P_{-}$ the projector on $\{h_i; \quad i\geq 3\}$, the negative subspace of the operator $\mathcal{L}$.

\section{Existence}
This section is devoted to the  proof of  the existence of a solution $v$ of $(\ref{eq_v})$ such that 
\begin{equation}\label{2.17}
\displaystyle \lim_{s\to +\infty}\|v(s)\|_{W^{1, \infty}_\beta}=0.
\end{equation}
For simplicity in the notations, we give the proof only when $N=1$ (see the comment right before \eqref{decomp}). In the following, we fix $\beta$ satisfying \eqref{beta}.\\
To do so, we use the framework developed in \cite{MZ97}, \cite{TZ}, \cite{NZ2}.
We proceed in two steps: Assuming some technical results, we prove in the first step  the existence of a solution $v$ of $(\ref{eq_v})$ which  converges to $0$ in $W^{1, \infty}_\beta$. The second step is devoted to the proof of the technical details.\\
In what follows, we denote by $C$ a generic positive constant, depending only on $p$, $\mu$, $\beta$ and $K_0$. Note that $C$  does not depend on $A$ and $s_0$, the constants that will appear below.
\subsection{Proof of the existence}
Let us explain briefly the general ideas of the proof.
 First, we define a shrinking set $\mathcal{V}_{\beta, K_0, A}(s)$ and translate our goal of making $v(s)$ go to $0$ in $W^{1, \infty}_\beta$ in terms of belonging to $\mathcal{V}_{\beta, K_0, A}(s)$. Reasonably, we choose the initial data such that it starts in  $\mathcal{V}_{\beta, K_0, A}(s_0)$.
 Using the spectral properties of equation $(\ref{eq_v})$, we reduce the problem
from the control of all  the five  components of $v$  (shown in \eqref{decomp}) in $\mathcal{V}_{\beta, K_0, A}(s)$ to  the control of its two first components $(v_0, v_1)$. That is, we reduce an infinite dimensional problem to a  finite dimensional one. Finally, we solve the finite dimensional problem  using index theory.
\subsubsection{Definition of a shrinking set $\mathcal{V}_{\beta, K_0, A}(s)$ and preparation of initial data}
Let first introduce the shrinking set as follows:
\begin{defi}\label{def_shrinking} (A set shrinking to zero) For all  $K_0\ge 1$, $A\geq 1$ and  $s\geq 1$, we define $\mathcal{V}_{\beta, K_0, A} (s)$ (or $\mathcal{V}_{A}(s)$ for simplicity) as the set of all functions  $g$ such that $(1+|y|^\beta)g\in L^\infty(\R^N)$  and
\begin{eqnarray}
|g_k(s)|\leq \displaystyle \frac{A}{s^2},\; k=0, 1, \;\;  |g_2(s)|\leq \displaystyle \frac{A^2 \log s}{s^2}, \;\; \|\displaystyle \frac{g_-(s)}{1+|y|^3} \|_{L^\infty}\leq  \displaystyle \frac{A}{s^2},
\end{eqnarray}
\begin{eqnarray}\label{ge}
 \|\displaystyle g_e(s) \|_{L^\infty}\leq  \displaystyle \frac{A^2}{\sqrt{s}}
, \;\;   \|\displaystyle (1+|y|^\beta) g_e(s) \|_{L^\infty}\leq  \displaystyle \frac{A^2}{s^{\frac{1-\beta}{2}}}.
\end{eqnarray}
\end{defi}
\begin{rem}\label{rema} Note from \eqref{hyp} and \eqref{beta} that $\beta <\frac{2}{p-1}<1$.
\end{rem}
Note that the shrinking set is different from all the previous studies. Indeed, it involves a new decay estimate at infinity in \eqref{ge}. Therefore, more estimates are needed.  Since $A\geq 1$, we remark that  the set  $\mathcal{V}_{\beta, K_0,  A}(s)$ is increasing (for fixed $s, \beta, K_0$) with respect to $A$ in the sense of inclusion. We also claim the following property of $\mathcal{V}_{ A}(s)$ (see Proposition \ref{prop_shrin} below for the proof):\\
For all $K_0\ge 1$, $A\geq 1$,  $\exists s_{01}(K_0, A)>0$ such that for all $s\geq s_{01}(K_0, A)$  and $g\in \mathcal{V}_{ A}(s)$, we have
\begin{equation}\label{v_e}
\|\displaystyle (1+|y|^\beta) g(s) \|_{L^\infty}\leq  \displaystyle \frac{CA^2}{s^{\frac{1-\beta}{2}}},
\end{equation}
\begin{equation}\label{v}
\|g(s) \|_{L^\infty}\leq  \displaystyle \frac{CA^2}{\sqrt{s}}.
\end{equation}
In other words, when the solution is in $\mathcal{V}_A(s)$, it is small both in the $L^\infty$ norm and in the weighted $L^\infty$ norm shown in \eqref{v} and  \eqref{v_e}. In fact, thanks to parabolic regularity, we will do more and show that the gradient of
the solution is also small in those norms (see Proposition \ref{prop-reg} below). This way, we may naturally translate our goal in \eqref{2.17} to the construction of a solution $v$ to equation \eqref{eq_v} satisfying
$$ v(s) \in {\cal V}_A(s), \; \mbox{for all}\;  s\ge s_0,$$
for some $s_0 \ge 1$.\\
The construction of a solution $v$ in $\mathcal{V}_{ A}(s)$ is based on a careful choice of the initial data at a time $s_0$. Let us consider the initial data as follows:
\begin{defi}(Choice of the initial data) For all  $K_0\ge 1$,  $A\geq 1$, $s_0\ge 1$ and $d_0, \, d_1\in \R$, we consider the following function as initial data for equation $(\ref{eq_v})$:
\begin{equation}\label{ci}
\psi_{s_0, d_0, d_1}(y)=\displaystyle \frac{A}{s_0^2}(d_0h_0(y)+d_1h_1(y))\chi(2y, s),
\end{equation}
where $h_i$, $i=0,1$ are defined in $(\ref{vect_prop})$ and $\chi$ is defined in $(\ref{chi})$.
\end{defi}

Thus, a natural question arises: Can we choose the initial data such that it starts in $\mathcal{V}_{ A}(s_0)$? For this end, we select the parameter $(d_0, d_1)$ as follows:
\begin{pro}\label{prop_ci}
(Properties of initial data) For each $K_ 0\ge 1$, $A\geq 1$, there exists $s_{02}(K_0, A)>1$ such that for all $s_0\geq s_{02}(K_0, A)$:
\begin{itemize}
\item[i)] There exists a rectangle $\mathcal{D}_{s_0}\subset[-2, 2]^2$ such that the mapping
\begin{equation}
\begin{array}{lcl}

\Phi : \R^2&\to &\R^2\\
(d_0, d_1)&\mapsto&(\psi_0, \psi_1),
\end{array}
\end{equation}
(where $\psi:=\psi_{s_0, d_0, d_1}$) is linear, one to one from $\mathcal{D}_{s_0}$ onto $[\displaystyle -\frac{A}{s_0^2},\displaystyle \frac{A}{s_0^2}]^2$ and maps  $\partial\mathcal{D}_{s_0}$ into $\partial([\displaystyle -\frac{A}{s_0^2},\displaystyle \frac{A}{s_0^2}]^2)$. Moreover, it has degree one on the boundary.\\
\item[ii)] For all $(d_0, d_1)\in \mathcal{D}_{s_0} $, $\psi \in \mathcal{V}_{ A}(s_0)$ with strict inequalities except for $(\psi_0, \psi_1)$, in the sense that
\begin{eqnarray}
\psi_e \equiv 0,\quad  |\psi_-(y)|< \frac{1}{s_0^2}(1+|y|^3),\;  \forall y\in \R,\\
|\psi_k|\leq \frac{A}{s_0^2},\;  k=0,1, \quad |\psi_2|< \displaystyle \frac{ \log s_0}{s_0^2}.
\end{eqnarray}
\item[iii)] Moreover,  for all $(d_0, d_1)\in  \mathcal{D}_{s_0}$, we have
\begin{eqnarray}\label{n_ci}
\|(1+|y|^\beta)\nabla \psi\|_{L^\infty}\leq \frac{CA}{s_0^{2-\frac{\beta}{2}}}\leq \frac{1}{s_0^{\frac{1-\beta}{2}}},\\
|\nabla \psi_-(y)| \leq \frac{1}{s_0^2}(1+|y|^3).
\end{eqnarray}
\end{itemize}
\end{pro}

The proof of the previous proposition follows exactly as in \cite{TZ} except for $(\ref{n_ci})$ (we insist  on the fact that the proof of \cite{TZ} holds here, even though the cut-off between $v_b$ and $v_e$ in \cite{TZ} holds at $y\sim s^{\frac{p+1}{2(p-1)}}$  instead of $y \sim \sqrt s$ here).  Indeed, the new condition  \eqref{ge} we have in the  shrinking set has no influence,  since  it involves $\psi_e$ and $\psi_e\equiv 0$ by construction in $(\ref{ci})$. That is the reason why the proof  is omitted except for  $(\ref{n_ci})$. The interested reader can find details in pages $5915-5918$ of \cite{TZ}. Thus, we only prove $(\ref{n_ci})$ below in section 3.2.1 page \pageref{sect}.\\

The following proposition is crucial in the proof of the existence of the blow-up solution. We reduce the problem to a finite dimensional problem. As in \cite{MZ}, \cite{EZ} and  \cite{TZ}, we prove that it is enough to control the $2$ components  $(v_0(s), v_1(s))\in [\displaystyle -\frac{A}{s^2},\displaystyle \frac{A}{s^2}]^2$ in order to control the solution $v(s)$ in $\mathcal{V}_{ A}(s) $, solution which is infinite dimensional.
\begin{pro}\label{prop_red} There exists $K_3\ge 1$ such that for any $K_0\ge K_3$, there exists $A_3(K_0) \geq 1$ such that for each
$A \geq A_3$, there exists $s_{03}(K_0, A)\in\R$ such that for all $s_0\geq s_{03}(K_0, A)$, the following holds:\\
If $v$ is a solution of $(\ref{eq_v})$ with initial data at $s = s_0$ given by $(\ref{ci})$ with $(d_0, d_1) \in  \mathcal{D}_{s_0}$, and
$v(s)\in \mathcal{V}_{ A}(s)$ for all $s \in  [s_0, s_1]$, with $v(s_1) \in  \partial \mathcal{V}_{ A}(s_1)$ for some $s_1 \geq s_0$, then:
\begin{itemize}
\item[i)](Reduction to a finite dimensional problem) We have:
$$(v_0(s_1), v_1(s_1))\in \partial([\displaystyle -\frac{A}{s_1^2},\displaystyle \frac{A}{s_1^2}]^2).$$
\item[ii)](Transverse crossing) There exist $m \in \{0, 1\}$ and $\omega\in \{-1, 1\}$ such that
$$\omega v_m(s_1) =\displaystyle \frac{A}{s_1^2}\quad \mbox{and}\quad
\omega v'_m(s_1) > 0.$$
\end{itemize}
\end{pro}
We give the proof of Proposition \ref{prop_red} in subsection \ref{sub_red}.\\
We remark by $(\ref{v_e})$ that  if a solution $v$ stays in $\mathcal{V}_{ A}(s)$, for $s\geq s_0$, then $\displaystyle (1+|y|^\beta) v(s)$ goes to $0$ in $L^\infty$.
As mentioned above, our goal is to get   the convergence in $W^{1, \infty}_\beta$. Therefore, it remains to show that $\|(1+|y|^\beta)\nabla v\|_{L^\infty}\to_{s\to \infty}0$. Thus,  we need the following parabolic regularity estimate for equation $(\ref{eq_v})$:
\begin{pro}\label{prop-reg}(Parabolic regularity in $\mathcal{V}_{ A}(s)$)\\
For all $K_0\ge 1$, $A\geq 1$, there exists  $s_{04}(K_0, A)$ such that for all $s_0\geq s_{04}(K_0, A)$, if $v$ is   the solution of $(\ref{eq_v})$  for all $s\in [s_0, s_1]$, $s_0\leq s_1$, with  $v(s)\in \mathcal{V}_{ A}(s)$ and  initial data at $s_0$, given in $(\ref{ci})$ with  $(d_0, d_1)\in \mathcal{D}_{s_0}$ defined in Proposition \ref{prop_ci}, then, for all $s\in[s_0, s_1]$, we have
\begin{eqnarray}\label{reg}
\displaystyle  \| \nabla v(s) \|_{L^\infty}\leq   \frac{CA^2}{\sqrt{s}} \quad \mbox{and}\quad \| (1+|y|^\beta)\nabla v(s) \|_{L^\infty}\leq  \displaystyle \frac{CA^2}{s^{\frac{1-\beta}{2}}}.
\end{eqnarray}
\end{pro}
The proof of the previous proposition is postponed to subsection \ref{sub_reg}.
\subsubsection{Proof of the existence of a solution  in $\mathcal{V}_{ A}(s)$}
We are going to prove  the following existence result  using the previous subsections.
\begin{pro}\label{exis}
There exists $K_5\ge 1$ such that  for all $K_0\ge K_5$, there exists $A_5(K_0)\geq 1$ such that for all $A\geq A_5(K_0)$ there exists $s_{05}(K_0,A)$ such that for all $s_0\geq s_{05}(K_0, A)$, there exists $(d_0, d_1)$ such that if $v$ is the solution of $(\ref{eq_v})$ with initial data at $s_0$, given in $(\ref{ci})$, then $v(s)\in \mathcal{V}_{ A}(s)$, for all $s\geq s_0$.
\end{pro}
\begin{proof}
Let us consider $K_0\ge K_3$,  $A\ge A_3(K_0)$, then   fix  $s_0 \geq max(s_{01},s_{02}, s_{03})$ and  $(d_0, d_1)\in  \mathcal{D}_{s_0}$, where the various constants are introduced in Section 3.1.1. This means that we can apply all the statements
in that section. The problem $(\ref{eq_v})$ with initial data at $s=s_0$, $\psi_{s_0, d_0, d_1}$ given in $(\ref{ci})$ has a solution $v_{d_0, d_1}(s)$ or  $v(s)$ for simplicity. Indeed, using a fixed point argument, we prove the wellposedness for equation $(\ref{eq_u})$ in
$W_\beta^{1,\infty}(\R^N)$ (we leave the proof to Appendix C).\\
According to Proposition  \ref{prop_ci}, for each $(d_0, d_1)\in  \mathcal{D}_{s_0}$,  $\psi_{s_0, d_0, d_1}\in \mathcal{V}_{ A}(s_0)\subset \mathcal{V}_{A+1}(s_0)$ and from the existence theory, starting  in $\mathcal{V}_{ A}(s_0)$, the solution $v(s)$ stays
in $\mathcal{V}_{ A}(s)$ until some maximal time $s_* = s_*(d_0, d_1)$. We  proceed  by contradiction and assume  that $s_*(d_0, d_1)<\infty$ for any
$(d_0, d_1) \in  \mathcal{D}_{s_0}$. By  definition of $s_*$, the solution at time  $s_*$ is on the boundary
of $\mathcal{V}_{ A}(s_*)$ and $v(s)\in \mathcal{V}_{ A}(s)$, for all $s\in [s_0, s_*]$.\\By Proposition \ref{prop_red}, we see  that $v(s_*)$ can leave $\mathcal{V}_{ A}(s_*) $ only by its first components, in other words, $(v_0(s_*), v_1(s_*))\in \partial([\displaystyle -\frac{A}{s_*^2},\frac{A}{s_*^2}]^2)$ and the following function is well defined
\begin{equation*}
\begin{array}{lcl}

\Phi :  \mathcal{D}_{s_0} &\to &\partial([-1, 1]^2)\\
(d_0, d_1)&\mapsto& \displaystyle \frac{s_*^2}{A}(v_0, v_1)(s_*).
\end{array}
\end{equation*}
Using the transversality property of $(v_0, v_1)$ given in Proposition $\ref{prop_red}$ part $ii)$, we prove that $s_*(d_0, d_1)$ is continuous. Therefore, $\Phi$ is continuous.\\
From Proposition \ref{prop_ci}, we have that if $(d_0, d_1)\in \partial\mathcal{D}_{s_0} $ then  $v(s_0)\in \mathcal{V}_{ A}(s_0) $, $(v_0(s_0), v_1(s_0))\in  \partial([\displaystyle -\frac{A}{s_0^2},\frac{A}{s_0^2}]^2)$ and we have strict inequalities for the other components.\\
Applying the transverse crossing property of  ii) of Proposition \ref{prop_red}, we see that $s_*(d_0,d_1) = s_0$, hence $(v_0(s_*), v_1(s_*)) = (\psi_0, \psi_1)$. Using  $i)$ in Proposition \ref{prop_ci}, we have that the restriction of $\Phi$ to the boundary is of degree $1$.\\
We conclude that $\Phi$ is continuous and  has  degree $1$ on the boundary. Therefore, we have a contradiction from the degree theory.
Thus, there exists a value $(d_0, d_1)\in \mathcal{D}_{s_0} $ such that for all $s\geq s_0$, $v(s)\in \mathcal{V}_{ A}(s) $.
This finishes the proof of Proposition \ref{exis}, assuming all the technical details.
\end{proof}
Since $v(s)\in \mathcal{V}_{ A}(s)$, we clearly see from $(\ref{v_e})$ and  $(\ref{reg})$ that
\begin{eqnarray}
\|\displaystyle (1+|y|^\beta) v(s) \|_{L^\infty}+\|\displaystyle (1+|y|^\beta)\nabla v(s) \|_{L^\infty}\leq  \displaystyle \frac{CA^2}{s^{\frac{1-\beta}{2}}}.
\end{eqnarray}

\subsection{Proof of the technical results}
In this section, we prove the technical results used in the previous section. As already written right before \eqref{decomp},  for simplicity in the notation, we give the proof in one dimension ($N=1$). We proceed in $4$ steps:
\begin{itemize}
\item In the first step, we prove estimates \eqref{v_e} and \eqref{v}, as well as  estimate $(\ref{n_ci})$ of  Proposition \ref{prop_ci}.
\item In the second step, we prove that if  $v(s)\in \mathcal{V}_{ A}(s)$, then  $B(v)$, $R(y,s)$ and $N(y,s)$ given in $(\ref{eq_v})$ are trapped in  $ \mathcal{V}_{ C}(s)$ and the potential term $Vv(s)\in   \mathcal{V}_{ CA}(s)$, for some positive constant $C$.
\item In the third step,  we prove the parabolic regularity result (Proposition \ref{prop-reg}).
\item In the last step, we prove the result of the reduction to a finite dimensional problem (Proposition \ref{prop_red}).
\end{itemize}
\subsubsection{Preparation of the initial data }\label{sect}
In this subsection, we prove estimates \eqref{v_e} and \eqref{v}, as well as   estimate $(\ref{n_ci})$ in Proposition \ref{prop_ci} and refer the reader to pages $5915-5918$ in \cite{TZ} for the other items of that proposition.\\
First, we give some properties of the shrinking set which imply estimates \eqref{v_e} and \eqref{v}:
\begin{pro}\label{prop_shrin}
For all $K_0\ge 1$, $A\geq 1$, there exists $s_2(K_0, A)$ such that, for all $s\geq s_2(K_0, A)$ and $g\in \mathcal{V}_{ A}(s)$, we have
\begin{itemize}
\item[i)] \begin{equation}\label{est_shrin}
\|g\|_{L^\infty(|y|\leq 2K_0\sqrt{s})}\leq \displaystyle \frac{CA}{\sqrt{s}}\quad \mbox{and}\quad \|g\|_{L^\infty(\R)} \leq \displaystyle \frac{CA^2}{\sqrt{s}}.
\end{equation}
\item[ii)] \begin{equation}
\|(1+|y|^\beta)g\|_{L^\infty(|y|\leq 2K_0\sqrt{s})}\leq \displaystyle \frac{CA}{s^{\frac{1-\beta}{2}}}\quad \mbox{and}\quad \|(1+|y|^\beta)g\|_{L^\infty(\R)} \leq \displaystyle \frac{CA^2}{s^{\frac{1-\beta}{2}}}.
\end{equation}
\end{itemize}
\end{pro}
\begin{proof}
  Property $i)$ follows exactly as in \cite{TZ} even though the cut-off occurs at $y\sim s^{\frac{p+1}{2(p-1)}}$ in \cite{TZ}, and at $y\sim \sqrt s$ here. We refer the reader to Proposition 4.7 of  \cite{TZ} page $5915$.\\
The first inequality of  $ii)$ follows from  estimate $(\ref{est_shrin})$. For the second inequality of $ii)$, if $|y|\le 2K_0 \sqrt s$, we just use the first inequality. If $|y|\ge 2K_0 \sqrt s$,
we simply note that $g=g_e$ and use the last item in the definition of ${\cal V}_A(s)$.
This concludes the proof of Proposition \ref{prop_shrin}.
\end{proof}
In the following, we prove  estimate $(\ref{n_ci})$ in  Proposition \ref{prop_ci}.
\begin{proof}[\textbf{Proof of  Proposition \ref{prop_ci}}]
 Since the initial data outside the blow-up area satisfies  $\psi_e=0$, we refer the reader  to page $5917$ of \cite{TZ} for the proof of $i), \; ii)$, except for  $(\ref{n_ci})$,  for which we give the details.
By the definition \eqref{ci} of initial data   and \eqref{vect_prop} of $h_0$, $h_1$, we see that
$$\nabla \psi(y)= d_1 \frac{A}{s_0^2}\chi(2y, s_0)+  \frac{A}{s_0^2}(d_0+d_1y)\displaystyle \chi'_0(\frac{2y}{K_0\sqrt{s_0}})\frac{2}{K_0\sqrt{s_0}}.$$
Since $supp(\psi)\subset \{ |y|< 2K_0 \sqrt{s_0}\}$ and $\|z\chi_0'(z)\|_{L^\infty}$, $\frac{2}{K_0\sqrt{s_0}} $ are bounded, recalling that $(d_0, d_1) \in {\cal D}_{s_0}\subset [-2,2]^2$,  we see that  for $s_0$ large enough
$$\displaystyle \|(1+|y|^\beta)\nabla\psi(y)\|_{L^\infty}\leq C\frac{A}{s_0^{2-\frac{\beta}{2}}}\leq \frac{1}{s_0^{\frac{1-\beta}{2}}}.$$
This concludes the proof of Proposition \ref{prop_ci}.
\end{proof}
\subsubsection{Preliminary estimates on various terms of equation $(\ref{eq_v})$}
In this subsection, we give various estimates on different terms appearing in  equation  $(\ref{eq_v})$. In particular, we prove  that  for $s$ large enough and some $C>0$,  the rest term $R(y, s)$ is  in $ \mathcal{V}_{ C}(s)$. We also  prove that  if $v(s)\in  \mathcal{V}_{ A}(s)$, then  the nonlinear term $B(v)\in  \mathcal{V}_{ C}(s)$ and the potential term $Vv$ is in $ \mathcal{V}_{ CA}(s)$. In addition, we prove that the new term $N(y,s)$ is trapped in $ \mathcal{V}_{ C}(s)$ under some additional assumptions on $v$.\\

\begin{lem}\label{lem_est}
\begin{enumerate}
\item For all $K_0\ge 1$, there exists $s_3(K_0)$ sufficiently large such that for $s\geq s_3(K_0)$,    the rest term
 $R\in  \mathcal{V}_{ C}(s)$.
\item For all $K_0\ge 1$ and $A\geq 1$, there exists $s_4(K_0, A)$ sufficiently large such that for $s\geq s_4(K_0, A)$, if $v(s)\in  \mathcal{V}_{ A}(s)$, then  the nonlinear term $B(v)\in  \mathcal{V}_{ C}(s)$ and the potential term $Vv \in \mathcal{V}_{ CA}(s)$.
\end{enumerate}
\end{lem}
\begin{proof}With respect to earlier papers using this method, this estimate is not new, except for two features:\\
- first, our shrinking set involves a new weighted estimate on the outer part of the solution (see the second item in  \eqref{ge}),
and naturally, we will prove that estimate for the various terms in this lemma;\\
 - second, our profile in \eqref{n_profile} involves a cut-off, which may change some estimates.\\
In fact, only the first issue is delicate, unlike the second.\\
Let us immediately say that the second question is not an issue, so we refer the reader to  subsection 4.2.2  page $5918-5923$ in  \cite{TZ} for the proof of all the 
items in Definition \ref{def_shrinking} of the shrinking set, and focus only on the proof of the second item in  \eqref{ge} for the various terms.
As we stated before, the fact that the cut-off occurs at different places in \cite{TZ} and in our paper doesn't affect the proofs
 at all.  
\begin{enumerate}
\item Estimate on the rest term  $R(y, s)$.\\
Since $-\frac{1}{2}zf'(z)-\frac{1}{p-1}f(z)+f(z)^p=0$ by definition \eqref{profile} of $f(z)$, we write $R(y, s)$ as follows
\begin{eqnarray*}
R(y,s)&=&\displaystyle \frac{1}{s} f"(z)+\frac{1}{2s}zf'(z)\\
&+& (f(z)+\chi_0(Z)\frac{\kappa}{2ps})^p-f(z)^p\\
&+& \frac{\kappa}{2ps}\Big[ \frac{1}{g_\epsilon^2}\chi_0"(Z)-(\frac{1}{2}-\frac{g_\epsilon'(s)}{g_\epsilon(s)} )Z\chi_0'(Z)+(\frac{1}{s}-\frac{1}{p-1})\chi_0(Z)\Big]\\
&=& R_i+R_{ii}+R_{iii},
\end{eqnarray*}
where $z=\frac{y}{\sqrt{s}}$, $Z=\frac{y}{g_\epsilon(s)}$, $g_\epsilon(s)=s^{\frac{1}{2}+\eps}$ and $\eps$ is fixed in $(0,\min(1, \frac{p-1}{4}))$. Note that $|y|\geq K_0\sqrt{s}$ on the support of $R_e$.\\
By definition $(\ref{profile})$ of $f$, we have
\begin{eqnarray}\label{f}
&f(z)&\!=\!(p-1+ \frac{(p-1)^2}{4p}z^2)^{-\frac{1}{p-1}}\!\sim_{z\to \infty}\! (\frac{(p-1)^2}{4p})^{-\frac{1}{p-1}}z^{-\frac{2}{p-1}}\\
&f'(z)&\sim_{z\to \infty}\frac{-2}{p-1} (\frac{(p-1)^2}{4p})^{-\frac{1}{p-1}}z^{-\frac{p+1}{p-1}}\\
&f"(z)&\sim_{z\to \infty}2\frac{p+1}{(p-1)^2} (\frac{(p-1)^2}{4p})^{-\frac{1}{p-1}}z^{-\frac{2p}{p-1}}.
\end{eqnarray}
In particular, if $|z|\geq K_0$, then
$$|R_i|\leq \frac{C}{s|z|^{\frac{2}{p-1}}}.$$
Hence, since $\beta<\frac{2}{p-1}$
\begin{eqnarray}\label{R_1}
(1+|y|^\beta)|R_i|\leq \frac{C}{s^{1-\frac{\beta}{2}}}\leq \frac{C}{s^{\frac{1-\beta}{2}}}.
\end{eqnarray}
On the other hand, we write
$$R_{ii}=f(z)^{p}\big((1+\frac{\kappa \chi_0(Z)}{2ps f(z)})^p-1\big).$$
Using $(\ref{f})$, we get
$$\frac{1}{s}\frac{\chi_0(Z)}{f(z)}\sim_{z\to \infty} C\frac{ z^{\frac{2}{p-1}}}{s}\chi_0(z\frac{\sqrt{s}}{g_\epsilon(s)})$$
uniformly in $s\ge 1$.\\
Since $\chi_0(z\frac{\sqrt{s}}{g_\epsilon(s)})$ is supported in  $\{ |z|\leq 2\frac{g_\epsilon(s)}{\sqrt{s}}=2s^\eps\}$, we may assume that $|z|\leq 2 \frac{g_\eps(s)}{\sqrt{s}}=2 s^\eps$, hence
$$\frac{1}{s}\frac{\chi_0(Z)}{f(z)}\leq \frac{C}{s^{1-\frac{2\eps}{p-1}}}.$$
Therefore,
$$\Big|(1+\frac{\kappa \chi_0(Z)}{2ps f(z)})^p-1\Big |\leq  \frac{C}{s^{1-\frac{2\eps}{p-1}}}.$$
Moreover, using the fact that $\eps \leq \frac{p-1}{4}$, we write
\begin{eqnarray}\label{R_2}
(1+|y|^\beta)|R_{ii}|\leq \frac{C}{s^{\frac12-\frac{2\eps}{p-1}}}\frac{1}{s^{\frac{1-\beta}{2}}}\leq \frac{C}{s^{\frac{1-\beta}{2}}}.
\end{eqnarray}
Finally, for the last term $R_{iii}$, since $\chi_0(Z)$, $Z\chi'_0(Z)$  and  $\chi''_0(Z)$  are bounded, we see that
$$|R_{iii}|\leq \frac{C}{s}.$$
Hence, since  $\eps\leq \frac{p-1}{4}<\frac{1}{2\beta}$ and  $R_{iii}$ is supported in
 $\{|y|\le 2K_0 g_\epsilon(s)\}$, we obtain \begin{eqnarray}\label{R_3}
(1+|y|^\beta)|R_{iii}|\leq \frac{C}{s^{\frac12-\eps\beta}}\frac{1}{s^{\frac{1-\beta}{2}}}\leq \frac{C}{s^{\frac{1-\beta}{2}}}.
\end{eqnarray}
Collecting all these bounds yields the bound for $R_e(s)$ as follows
 \begin{eqnarray}
(1+|y|^\beta)|R_e(s)|\leq \frac{C}{s^{\frac{1-\beta}{2}}}.
\end{eqnarray}
\item The nonlinear term $B(v)$.\\
Since $p>3$, using the definition \eqref{nonlinear-term} of $B(v)$ and a Taylor expansion, we write
$$|B(v)|\leq C |v|^2.$$
Hence 
 $$|B_e(v)|=(1-\chi)|B(v)|\leq C |v| |v_e|. $$
 From the fact that $v\in  \mathcal{V}_{ A}(s)$, using \eqref{ge} and \eqref{v},  we have for $s$ large enough
 $$(1+|y|^\beta)|B_e(v)|\leq C \frac{A^2}{s^{\frac{1-\beta}{2}}} \frac{CA^2}{\sqrt{s}}\leq \frac{1}{s^{\frac{1-\beta}{2}}}.$$
 \item The potential term $vV$.\\
 By definition \eqref{linear} of $V$,  we see that, 
 $$\|V(s)\|_{L^\infty}\leq C.$$
Then, using the fact that $v\in \mathcal{V}_{ A}(s)$, we get,
$$(1+|y|^\beta)|(Vv)_e|\leq \|V\|_{L^\infty}(1+|y|^\beta)|v_e|\leq  C \frac{A^2}{s^{\frac{1-\beta}{2}}} .$$
\end{enumerate}
This  concludes the proof of Lemma \ref{lem_est}.
\end{proof}
We now estimate the new term defined in \eqref{new-term}. We claim the following Proposition:
\begin{pro}\label{prop_new_term} For all $K_0\ge 1$, $A\geq 1$, there exists $s_5(K_0, A)$ sufficiently large, such that for all $s\geq s_5(K_0, A)$, if $v(s)\in \mathcal{V}_{ A}(s)$ is such that
\begin{equation}\label{reg_v}
\|\nabla v(s)\|_{L^\infty}\leq \frac{CA^2}{\sqrt{s}}, \quad  \|(1+|y|^\beta)\nabla v(s)\|_{L^\infty}\leq \frac{CA^2}{s^{\frac{1-\beta}{2}}},
\end{equation}
then the new term defined in \eqref{new-term} satisfies $$|N_i(s)|\le Ce^{-\frac \gamma2 s}, \,0\le i \le 2, \quad \displaystyle \|\frac{N_{-}(y,s)}{1+|y|^3}\|_{L^\infty}\leq  Ce^{-\frac \gamma2 s} $$ and    $N\in \mathcal{V}_{ C}(s)$, for some positive constant $C$.
\end{pro}
Before proving this Proposition, we need the following Lemma:
\begin{lem}\label{lem_new_term} Under  the assumption of Proposition \ref{prop_new_term}, we have, for $s$ sufficiently large
\begin{enumerate}
\item $\|N(y,s)\|_{L^\infty}\leq Ce^{-\frac{\gamma}{2}s}$,
\item $\|(1+|y|^\beta)N_e(y,s)\|_{L^\infty}\leq Ce^{-\frac{\gamma}{2}s}$,
\end{enumerate}
where $C$ is a positive constant.
\end{lem}
\begin{proof}~\\
 -\textit{Proof of 1}. By definition \eqref{new-term} of $N(y,s)$, this lemma is relevant only when $\mu \neq 0$.\\
Recall  from \eqref{n_profile} and \eqref{new-term} that 
\begin{equation}\label{new_term_int}
 N(y, s)=\mu e^{-\gamma s}\nabla(v+\varphi)\int_{B(0,|y|)} |v+\varphi|^{q-1},
\end{equation}
where  $\gamma=\frac{p-q}{p-1}>0$, $\varphi(y,s)=\displaystyle f(\frac{y}{\sqrt{s}})+\frac{\kappa}{2ps}\chi_0(\frac{y}{g_\epsilon(s)})$ and 
\begin{equation}\label{new_profil_nabla}
\displaystyle \nabla \varphi = -\frac{p-1}{2p}\frac{1}{\sqrt{s}}(\frac{y}{\sqrt{s}}f^p(\frac{y}{\sqrt{s}}))+\frac{\kappa}{2psg_\epsilon(s)}\chi'_0(\frac{y}{g_\epsilon(s)}).
\end{equation} 
Since $zf^p(z)$ and $\chi'_0(z)$ are bounded, we get
\begin{equation}\label{nabla_phi}
\|\nabla\varphi\|_{L^\infty}\leq \frac{C}{\sqrt{s}}.
\end{equation}
Therefore, using \eqref{reg_v}, we write 
\begin{equation}\label{nabla_v+phi}
\|\nabla(\varphi+v)\|_{L^\infty}\leq \frac{CA^2}{\sqrt{s}}.
\end{equation}
Now, by definition \eqref{n_profile} of $\varphi$, we write
\begin{eqnarray*}
\int_{B(0,|y|)}|v+\varphi|^{q-1} 
&\le &\int_{\R^N}|v|^{q-1}+\int_{\R^N}|f(\frac{y'}{\sqrt s})|^{q-1}dy'
+\frac{C}{s^{q-1}}\int_{\R^N}|\chi_0(\frac{y'}{g_\epsilon(s)})|^{q-1}dy'\\
&\equiv &I_1+I_2+I_3.
\end{eqnarray*}
In the following, we will bound $I_1$, $I_2$ and $I_3$.\\
- \textit{Bound on $I_1$}. Since $v(s) \in {\cal V}_A(s)$, using item (ii) in Proposition \ref{prop_shrin}, we see that
$$
\forall y'\in \R^N,\;\;
|v(y',s)| \le \frac{CA^2}{s^{\frac{1-\beta}2}(1+|y'|^\beta)},
$$
hence
\begin{equation}\label{I1}
I_1
\le (\frac{CA^2}{s^{\frac{1-\beta}2}})^{q-1}\int_{\R^N} \frac{1}{1+|y'|^{\beta(q-1)}} dy'
\le (\frac{CA^2}{s^{\frac{1-\beta}2}})^{q-1},
\end{equation}
since $\beta(q-1)>N$ from the choice of $\beta$ in \eqref{beta}.\\
- \textit{Bound on $I_2$ and $I_3$}. By definition \eqref{profile} of $f$, we write after a change of variables
\begin{equation}\label{I2I3}
I_2 = s^{N/2}\int_{\R^N} f^{q-1} = Cs^{N/2}
\mbox{ and }
I_3 = C\frac{g_\epsilon(s)^N}{s^{q-1}}\int_{\R^N}\chi_0^{q-1}=C \frac{g_\epsilon(s)^N}{s^{q-1}}.
\end{equation}
since both integrals are convergent, because $\chi_0$ is compactly supported, and also thanks to the definition \eqref{profile} of $f$ and the  condition on $p$ in \eqref{hyp}.\\
Gathering \eqref{I1} and \eqref{I2I3}, we see that for $s$ large enough, we have
\begin{equation}\label{I}
\int_{B(0,|y|)}|v+\varphi|^{q-1}\le Cs^\alpha ,
\end{equation}
for some $\alpha>0$.\\
Using \eqref{new_term_int} and  \eqref{nabla_v+phi}, we get the conclusion of item 1 in this lemma, provided that $s$ is large enough.\\
 - \textit{Proof of 2}. From the definition \eqref{new_term_int} of $N(y,s)$, condition \eqref{reg_v}, estimate \eqref{nabla_phi} and estimate 
\eqref{I}, it is enough to show that

\begin{equation}\label{gphi}\forall |y|\ge K_0\sqrt s, \;\; |y|^\beta|\nabla \varphi(y,s)|\le \frac C{s^{\frac{1-\beta}2}}  
\end{equation}
in order to conclude.\\ 
Using \eqref{new_profil_nabla}, we have 
  $$|y|^\beta\nabla
 \varphi=|y|^\beta\Big(-\frac{2(p-1)}{4p}\frac{1}{\sqrt{s}}(\frac{y}{\sqrt{s}}f^p(\frac{y}{\sqrt{s}}))+\frac{\kappa}{2psg_\epsilon(s)} \chi_0'(\frac{y}{g_\epsilon(s)})\Big).$$ 
  We first write
  $$\big||y|^\beta\frac{1}{\sqrt{s}}(\frac{y}{\sqrt{s}}f^p(\frac{y}{\sqrt{s}}))\big|\leq  \frac{C}{s^{\frac{1-\beta}{2}}}(\frac{y}{\sqrt{s}})^{\beta -\frac{p+1}{p-1}} \leq  \frac{C}{s^{\frac{1-\beta}{2}}},  $$
 thanks to condition \eqref{beta} on $\beta$.\\
 By Remark \ref{rema} and  the definition of $\chi_0$ and $g_\epsilon$, we then  derive that $\beta<1$ and $|z|^\beta \chi_0(z)$ is bounded, hence
  $$||y|^\beta \frac{\kappa}{2psg_\epsilon(s)} \chi_0'(\frac{y}{g_\epsilon(s)})|\leq C\frac{g_\epsilon(s)^{\beta-1}}{s}\leq  \frac{C}{s^{\frac{1-\beta}{2}}},$$
and \eqref{gphi} follows. This concludes the proof of item $2$ and Lemma \ref{lem_new_term} too.
\end{proof}
We now give the proof of Proposition \ref{prop_new_term}.
\begin{proof}[\textbf{Proof of Proposition \ref{prop_new_term}}]\label{proof3.10}
By  definition of $N_m$, $0\leq m\leq 2$, we have from Lemma \ref{lem_new_term}:
\begin{eqnarray*}
|N_m(s)|=|\int_\R N(y,s) k_m(y)\rho dy|\leq Ce^{-\frac{\gamma}{2}s}\int_{\R}|k_m(y)| \rho dy\leq C_m e^{-\frac{\gamma}{2}s}.
\end{eqnarray*}
Since $\gamma>0$ and $s$ is  sufficiently large, we obtain
$$|N_m(s)|\leq \frac{C}{s^2}, \; \mbox{for}\; 0\leq m\leq 1,$$
and $$|N_2(s)|\leq C^2 \frac{\log s}{s^2}. $$
Moreover, we have that
$$|N_b(y,s)|=|N(y,s)\chi(y,s)|\leq  C|N(y,s)|\leq C e^{-\frac{\gamma}{2}s}.$$
Furthermore, by the definition of $N_{-}(y,s)$ (see \eqref{decomp} and following lines), we see that
\begin{eqnarray*}
|N_{-}(y,s)|=|N_b(y,s)-\sum_{m=0}^2N_m(s) h_m(y)|\leq Ce^{-\frac{\gamma}{2}s}(1+|y|^3).
\end{eqnarray*}
Hence,
$$\displaystyle \|\frac{N_{-}(y,s)}{1+|y|^3}\|_{L^\infty}\leq \frac{C}{s^2}.$$
On the other hand, for $s$ sufficiently large, we have
 $$|N_e(y,s)|=|N(y,s)(1-\chi(y,s))|\leq  2C e^{-\frac{\gamma}{2}s}\leq \frac{C^2}{\sqrt{s}}.$$
 Finally by Lemma \ref{lem_new_term}, for $s$ sufficiently large, we obtain
 $$|(1+|y|^\beta)N_e(y,s)|\leq  C e^{-\frac{\gamma}{2}s}\leq \frac{C^2}{s^{\frac{1-\beta}{2}}}.$$
 This finishes the proof of Proposition \ref{prop_new_term}.
\end{proof}
 \subsubsection{Parabolic regularity}\label{sub_reg}
In this subsection, we prove Proposition \ref{prop-reg}. The proof follows as in \cite{EZ} and  \cite{TZ}, with some additional care, since we have a different shrinking set and a different nonlinear  term. The proof relies mainly on some properties of the semi-group $e^{\theta \mathcal{L}}$:
\begin{lem}\label{lem_kernel1}
The Kernel $e^{\theta \mathcal{L}}(y,x)$ of the semi-group $e^{\theta \mathcal{L}}$ is given by
\begin{equation}\label{kernel_e}
e^{\theta \mathcal{L}}(y,x)=\displaystyle \frac{e^\theta}{\sqrt{4\pi(1-e^{-\theta})}}\exp[-\frac{(ye^{-\theta/2}-x)^2}{4(1-e^{-\theta})}],
\end{equation}
for all $\theta>0$, and $e^{\theta\mathcal{L}}$ is defined by
\begin{equation}
e^{\theta \mathcal{L}}r(y)= \int_\R e^{\theta \mathcal{L}}(y,x) r(x) dx.
\end{equation}
We have the following estimates:
\begin{enumerate}
\item If $r_1\leq r_2$, then $e^{\theta \mathcal{L}}r_1\leq e^{\theta \mathcal{L}}r_2$.
\item \begin{itemize}
\item[i)]If $r\in W^{1, \infty}(\R)$, then $\|\nabla( e^{\theta \mathcal{L}}r)\|_{L^\infty}\leq Ce^{\frac{\theta}{2}}\|\nabla r\|_{L^\infty}$.
\item[ii)] If $r\in L^\infty(\R)$, then $\|\nabla( e^{\theta \mathcal{L}}r)\|_{L^\infty}\leq \frac{Ce^{\frac{\theta}{2}}}{\sqrt{1-e^{-\theta}}}\|\ r\|_{L^\infty}$.
\end{itemize}
\item For $m\geq 0$, if $|r(x)|\leq \mu (1+|x|^m)$, $\forall x\in \R$, then
\begin{itemize}
\item[i)] $|e^{\theta \mathcal{L}}r(y)|\leq C\mu e^\theta (1+|y|^m)$, $\forall y\in \R$.
\item[ii)] $|\nabla(e^{\theta \mathcal{L}}r(y))|\leq C\mu \frac{e^\frac{\theta}{2}}{\sqrt{1-e^{-\theta}}} (1+|y|^m)$, $\forall y\in \R$.
\end{itemize}
\item For $m\geq 0$, if $|\nabla r(x)|\leq \mu (1+|x|^m)$, $\forall x\in \R$, then   $|\nabla(e^{\theta \mathcal{L}}r(y))|\leq C\mu e^\frac{\theta}{2} (1+|y|^m)$,  $\forall y\in \R$.
\item For $0<m<1$, we have
\begin{itemize}
\item[i)]  If $(1+|y|^m)r\in W^{1,\infty}(\R)$, then
 $$\|(1+|y|^m)\nabla (e^{\theta \mathcal{L}}r)(y)\|_{L^\infty}\leq C e^{\frac{(m+1)\theta}{2}} \|(1+|y|^m)\nabla r\|_{L^\infty}.$$
\item[ii)]  If $(1+|y|^m)r\in L^\infty(\R)$, then $$ \|(1+|y|^m)\nabla (e^{\theta \mathcal{L}}r)(y)\|_{L^\infty}\leq C \frac{e^\frac{(m+1)\theta}{2}}{\sqrt{1-e^{-\theta}}}  \|(1+|y|^m)r\|_{L^\infty}.$$
\end{itemize}
\end{enumerate}
\end{lem}
\begin{proof}
Because  estimates $1)-4)$ are not new, we refer the reader to Lemma $4.15$ page $5926$ in \cite{TZ}, See also  pages $554-555$ in \cite{bricmont}.\\
Thus, we only prove  $5)$. In order  to avoid unnecessary  technicalities here, we prove it  in  Appendix A.
\end{proof}
\bigskip
We are now going to prove  Proposition \ref{prop-reg}.
\begin{proof}[\textbf{Proof of Proposition \ref{prop-reg}}]
Both estimates in \eqref{reg} follow by the same technique as in Proposition 3.3 page 14 in \cite{EZ} and Proposition 4.17 page 5929 in \cite{TZ}. For that reason, we only prove the second, since it is a little more delicate.\\ 
 Let $K_0\ge 1$,  $A\geq 1$, $s_0\geq 1$ and  consider $v(s)$ a solution of equation (\ref{eq_v}) defined on $[s_0, s_1]$, where $s_1\geq s_0\geq 1$ and $v(s_0)=\psi$ defined in (\ref{ci}) with $(d_0, d_1)\in \mathcal{D}_{s_0}$ defined in Proposition \ref{prop_ci}. We assume in addition that $v(s)\in \mathcal{V}_{ A}(s)$, for all $s\in [s_0, s_1]$.\\
We distinguish two cases.\\
{\bf{\underline{Case 1:}}} If $s\leq s_0+1$. Let $s'_1=\min(s_0+1, s_1)$ and take $s\in [s_0, s'_1]$. Then we have for any  $t\in [s_0, s]$, $$s_0\leq t\leq s\leq s_0+1\leq 2s_0, \; \mbox{hence}\quad  \frac{1}{s}\leq \frac{1}{t}\leq \frac{1}{s_0}\leq \frac{2}{s}. $$
From equation $(\ref{eq_v})$, we write for any $s\in [s_0,s_1']$,
\begin{equation}\label{eq_int_v_bis}
 v(s)=e^{(s-s_0)\mathcal{L}}v(s_0)+\int_{s_0}^s e^{(s-t)\mathcal{L}} F(t)dt,
 \end{equation}
where $$F(y,t)=Vv(y,t)+B(v)+R(y,t)+N(y,t).$$
From Lemma \ref{lem_kernel1}, we see that for all $s\in [s_0, s'_1]$
\begin{eqnarray}\label{3.54}
\|(1\!+\!|y|^\beta)\nabla v(s)\|_{L^\infty}\!\!\!\!&\leq& \!\!\|(1\!+\!|y|^\beta)\nabla  e^{(s-s_0)\mathcal{L}}v(s_0)\|_{L^\infty}\!\!+\!\!\!\int_{s_0}^s\! \| (1\!+\!|y|^\beta)\nabla e^{(s-t)\mathcal{L}} F(t)\|_{L^\infty}dt\nonumber \\
\!\!&\leq &\!\! C \|(1\!+\!|y|^\beta)\nabla v(s_0)\|_{L^\infty}\!\!+\!C\!\int_{s_0}^s \frac{\|(1+|y|^\beta)F(t)\|_{L^\infty}}{ \sqrt{1-e^{-(s-t)}}}dt.
\end{eqnarray}
Using Proposition \ref{prop_ci}, we obtain
$$ \|(1+|y|^\beta)\nabla v(s_0)\|_{L^\infty}\leq \frac{CA}{s_0^{2-\frac{\beta}{2}}}.$$
Since $s\leq 2s_0$, we deduce that
\begin{equation}\label{3.54.5}
 \|(1+|y|^\beta)\nabla v(s_0)\|_{L^\infty}\leq \frac{CA}{s^{2-\frac{\beta}{2}}}.
\end{equation}
We now estimate $\|(1+|y|^\beta)F(t)\|_{L^\infty}$: we write
$$(1+|y|^\beta)F(y,t)=(1+|y|^\beta)(Vv(y,t)+B(v)+R(y,t))+(1+|y|^\beta)N(y,t).$$
From Lemma \ref{lem_est},  we see that $R$, $B(v)\in \mathcal{V}_{ C}(t)$, and $Vv\in \mathcal{V}_{ CA}(t)$.
Therefore by Proposition \ref{prop_shrin}, we have
\begin{equation}
\|(1+|y|^\beta)(Vv+B(v)+R)\|_{L^\infty}\leq C\frac{A^2}{t^{\frac{1-\beta}{2}}}.
\end{equation}
Furthermore, using  inequality \eqref{new_term_int}, \eqref{nabla_phi},  \eqref{gphi}  and  \eqref{I}, we see that
\begin{equation*}
\|(1+|y|^\beta)N\|_{L^\infty}\leq C e^{-\gamma t}\Big(\frac{C}{t^{\frac{1-\beta}{2}}}+\|(1+|y|^\beta)\nabla v\|_{L^\infty}\Big) t^\alpha,
\end{equation*}
for some $\alpha>0$.\\
If we assume $s_0$ large enough, and consider   $t\in [s_0, s]$, then we have

\begin{equation}
\|(1+|y|^\beta)N\|_{L^\infty}\leq C e^{-\frac{\gamma t}{2}}\Big(\frac{1}{t^{\frac{1-\beta}{2}}}+\|(1+|y|^\beta)\nabla v\|_{L^\infty}\Big).
\end{equation}
Collecting all these bounds  and using the fact that  $t\leq s \leq 2t$, we obtain   for all $t\in [s_0, s]$,
\begin{equation}\label{3.57}
\|(1+|y|^\beta)F(t)\|_{L^\infty}\leq C\frac{A^2}{s^{\frac{1-\beta}{2}}}+ C e^{-\frac{\gamma s}{2}}\|(1+|y|^\beta)\nabla v\|_{L^\infty}.
\end{equation}
Therefore, using $(\ref{3.54})$ and \eqref{3.54.5}, then taking $s_0$ large enough, we write 
\begin{equation*}
\|(1+|y|^\beta)\nabla v(s)\|_{L^\infty}\leq  C\frac{A}{s^{2-\frac{\beta}{2}}}+ C\frac{A^2}{s^{\frac{1-\beta}{2}}}+ C e^{-\frac{\gamma s}{2}}\int_{s_0}^s\frac{\|(1+|y|^\beta)\nabla v\|_{L^\infty}}{\sqrt{1-e^{-(s-t)}}}dt.
\end{equation*}
Using a Gronwall's argument, we obtain that for $s_0$ large enough
\begin{equation*}
\|(1+|y|^\beta)\nabla v(s)\|_{L^\infty}\leq  C\frac{A^2}{s^{\frac{1-\beta}{2}}}, \quad \forall s\in [s_0, s'_1].
\end{equation*}
{\bf{\underline{Case 2:}}} If $s> s_0+1$. Take $s\in ]s_0+1, s_1]$ and rewrite
equation $(\ref{eq_int_v_bis})$ for any $s'\in (s-1, s]$ (use the fact that  $s\geq s_0+1 \geq 2$ and $s=s-1+1\leq 2(s-1)$).\\
More precisely, we rewrite
$$ v(s')=e^{(s'-s+1)\mathcal{L}}v(s-1)+\int_{s-1}^{s'} e^{(s'-t)\mathcal{L}} F(t)dt.$$
From Lemma \ref{lem_kernel1}, we see that for all $s\in [s_0, s'_1]$
\begin{eqnarray}
\|(1\!+\!|y|^\beta)\nabla v(s')\|_{L^\infty}\!\!\!\!&\leq&\! \!\!\|(1\!+\!|y|^\beta)\nabla  e^{(s'\!-\!s\!+\!1)\mathcal{L}}v(s\!-\!1)\|_{L^\infty}\!\!+\!\!\int_{s\!-\!1}^{s'} \!\!\| (1\!+\!|y|^\beta)\nabla e^{(s'\!-\!t)\mathcal{L}} F(t)\|_{L^\infty}dt\nonumber \\
&\leq & \!\!C\frac{\|(1+|y|^\beta)v(s-1)\|_{L^\infty}}{\sqrt{1-e^{-(s'-s+1)}}} \!+\!C\int_{s-1}^{s'} \frac{\|(1+|y|^\beta)F(t)\|_{L^\infty}}{ \sqrt{1-e^{-(s'-t)}}}dt.
\end{eqnarray}
Since $v\in \mathcal{V}_{ A}$ and $\frac{s'}{2}<s-1<s'<s$, we have from Proposition \ref{prop_ci}
$$\|(1+|y|^\beta)v(s-1)\|_{L^\infty}\leq \frac{CA^2}{(s-1)^{\frac{1-\beta}{2}}}\leq \frac{CA^2}{(s')^{\frac{1-\beta}{2}}},$$
and from \eqref{3.57},
\begin{eqnarray*}\|(1+|y|^\beta)F(t)\|_{L^\infty}&\leq& C \frac{A^2}{t^{\frac{1-\beta}{2}}}+Ce^{-\frac{\gamma t}{2}}\|(1+|y|^\beta)\nabla v(t)\|_{L^\infty}\\
&\leq &  C \frac{A^2}{(s')^{\frac{1-\beta}{2}}}+Ce^{-\frac{\gamma s'}{2}}\|(1+|y|^\beta)\nabla v(t)\|_{L^\infty}.
\end{eqnarray*}
Therefore,
\begin{equation*}
\|(1+|y|^\beta)\nabla v(s')\|_{L^\infty}\leq  C\frac{A^2}{(s')^{\frac{1-\beta}{2}}\sqrt{1-e^{-(s'-s+1)}}}+ C e^{-\frac{\gamma s'}{2}}\int_{s-1}^{s'}\frac{\|(1+|y|^\beta)\nabla  v\|_{L^\infty}}{\sqrt{1-e^{-(s'-t)}}}dt.
\end{equation*}
Using a Gronwall's argument, we see that
\begin{equation*}
\|(1+|y|^\beta)\nabla v(s')\|_{L^\infty}\leq 2 C\frac{A^2}{(s')^{\frac{1-\beta}{2}}\sqrt{1-e^{-(s'-s+1)}}}, \quad \forall s'\in [s-1, s].
\end{equation*}
Taking $s'=s$, we conclude the proof of Proposition \ref{prop-reg}.
\end{proof}
\subsubsection{Reduction to a finite dimensional problem}\label{sub_red}
This subsection is crucial in the proof of our result. It is dedicated to the proof of  Proposition \ref{prop_red}.  In this subsection, we reduce the problem to a finite dimensional one. We prove through a priori estimates that the control of $v(s)$ in $\mathcal{V}_{ A}(s)$ is reduced to the control of $(v_0, v_1)(s)$ in $[-\frac{A}{s^2}, \frac{A}{s^2}]^2$.\\
 For this end,  we project equation $(\ref{eq_v})$ on the different components of the decomposition $(\ref{decomp})$ and we get new bounds on all components of $v$:
\begin{pro}\label{pro_red}
There exists $K_6$ such that for any $K_0\ge K_6$, there exists  $A_6\geq 1$ such that for all $A\geq A_6$, there exists $s_{0,6}(K_0, A)$ large enough such that the following holds for all $s_0\geq s_{0,6}(K_0, A)$:\\
Assume that for some $s_1\geq \sigma\geq s_0$, we have
$$ v(s)\in \mathcal{V}_{ A}(s), \quad \mbox{for all} \; s\in [\sigma, s_1], $$
and that $\nabla v$ satisfies the estimates stated in Proposition \ref{prop-reg}.
Then, the following holds for all $s\in [\sigma, s_1]$,
\begin{enumerate}
\item[i)] (ODE satisfied by the positive mode):
For $m\in \{0,1\}$, we have
$$\displaystyle|v'_m(s)-(1-\frac{m}{2}) v_m(s)|\leq \frac{C}{s^2}.$$
\item[ii)] (ODE satisfied by the null mode):
We have
$$|v'_2(s)+\frac{2}{s}v_2(s)|\leq\frac{C}{s^3}.$$
\item[iii)] (Control of the negative and outer modes): We have
\begin{eqnarray*}
\displaystyle \|\frac{v_{-}(s)}{1+|y|^3}\|_{L^\infty}\!\!\!&\leq&\!\!\! Ce^{-\frac{s-\sigma}{2}}\|\frac{v_{-}(\sigma)}{1+|y|^3}\|_{L^\infty}\!+\!C\frac{e^{-(s-\sigma)^2}}{s^{\frac{3}{2}}}\|v_e(\sigma)\|_{L^\infty}\!+ \!C\frac{1+s-\sigma}{s^2},\\
 \|v_{e}(s)\|_{L^\infty}\!\!\!\!\!\!&\leq& \!\!\!Ce^{-\frac{s-\sigma}{p}}\|v_e(\sigma)\|_{L^\infty}\!+ \!Ce^{s-\sigma}s^{\frac{3}{2}}\|\frac{v_{-}(\sigma)}{1+|y|^3}\|_{L^\infty}\!+ \!C\frac{1+(s-\sigma)e^{s-\sigma}}{\sqrt{s}}.
\end{eqnarray*}
\item[iv)](Control of the term outside the blow-up area in the new functional space):
\begin{eqnarray*}
 \displaystyle \|(1+|y|^\beta)v_{e}(s)\|_{L^\infty}\!\!\!&\leq& \!\!\!C \displaystyle e^{-\frac{s-\sigma}{2}(\frac{1}{p-1}-\frac{\beta}{2})} \|(1+|y|^\beta)v_{e}(\sigma)\|_{L^\infty}+ Ce^{\frac{\beta}{2}(s-\sigma)}s^{\frac{3}{2}+\frac{\beta}{2}}\|\frac{v_{-}(\sigma)}{1+|y|^3}\|_{L^\infty}\\&&+ C\frac{1+(s-\sigma)e^{\frac{\beta}{2}(s-\sigma)}}{s^{\frac{1-\beta}{2}}}.
\end{eqnarray*}
\end{enumerate}

\end{pro}
\begin{proof}
Note that estimates $i)-iii)$ are stated in \cite{TZ} for an equation lacking the new term $N$. Since the new term $N$ satisfies Proposition \ref{prop_new_term} and Lemma \ref{lem_new_term}, the reader can adapt easily the proof of \cite{TZ} to  the new situation. For this reason, we only prove the estimate $iv)$.\\
We write the integral form of equation \eqref{eq_v}:
\begin{eqnarray}\label{decomp1}
v(s)&=&K(s, \sigma) v(\sigma)+\int_\sigma^s K(s,t)(B(v(t))+R(t)+N(t))dt.\nonumber\\
&=& \mathcal{A}(s)+\mathcal{B}(s)+\mathcal{C}(s)+\mathcal{D}(s).
\end{eqnarray}
The proof is given in three steps: In the first step, we need to
understand the behavior of the Kernel $K(s, \sigma)$, which plays an
important role in estimating the new components of $v$. In the second
step, we use these estimates to give new bounds on different terms
appearing in $(\ref{decomp1})$. In the third step, we conclude the
proof.

\bigskip

{\bf{Step 1: Linear estimates on the operator $\mathcal L+V$.}}\\
It is clear that the kernel $K(s,\sigma)$ has  stronger influence   in this formula. For this reason,  it is convenient to give the following result of Bricmont and Kupiainen
 \cite{bricmont} (Note that estimate $4)$ is new, crucial and ours):

 \begin{lem}\label{lem_kernel2} We have the following estimates:
 \begin{enumerate}
 \item For all  $1\leq \sigma \leq s \leq 2\sigma$, $x, y\in \R$, $K(s, \sigma, y, x)\leq C e^{(s-\sigma)\mathcal{L}}(y,x)$.
 \item For all $m\geq 0$,  $1\leq \sigma \leq s \leq 2\sigma$, $y\in \R$,  we have
$$|\displaystyle\int K(s, \sigma, y,x)(1+|x|^m)dx|\leq Ce^{(s-\sigma)}(1+|y|^m).$$
 \item For all   $\eta>0$, there exists $K_1(\eta)$ such that for all $K_0\geq K_1(\eta)$,  for all $\sigma \ge 1$ and $s\in [\sigma, 2\sigma ]$ and   for all $y\in \R$, we have
 $$\displaystyle|\int K(s, \sigma, y,x){\bf{1}}_{\{|x|\geq K_0\sqrt{\sigma}\}} dx|\leq C e^{-(\frac{1}{p-1}-\eta)(s-\sigma)}.$$
 \item For all $0\leq m <\frac{2}{p-1}$,  there exists $K_2(m)$ such that for all $K_0\geq K_2(m)$, for all $\rho>0$ there exists $s_2(K_0, \rho)\ge 1$ such that for any $\sigma \ge s_2(K_0, \rho)$ and $s\in [\sigma, \sigma+\rho]$, we have
  $$\displaystyle|\int \frac{K(s, \sigma, y,x)}{1+|x|^m}{\bf{1}}_{\{|x|\geq K_0\sqrt{\sigma}\}} dx|\leq C\frac{e^{-\frac{1}{2}(\frac{1}{p-1}-\frac m2)(s-\sigma)}}{1+|y|^m}.$$
 \end{enumerate}
 \end{lem}
 We give the proof of the above Lemma in the Appendix B. Using this result, we obtain the following:
 \begin{lem}\label{lem_14} There exsits $K_{10}\ge 1$ such that for all $K_0 \ge K_{10}$, 
for all $\rho >0$, there exists $\sigma_0=\sigma_0(K_0,\rho)$ such that if $\sigma\geq \sigma_0\geq 1$ and $g(\sigma)$ satisfies
 $$\displaystyle \sum_{m=0}^2 |g_m(\sigma)|+\|\frac{g_{-}(y, \sigma)}{1+|y|^3}\|_{L^\infty}+\|( 1+|y|^\beta)g_e(\sigma)\|_{L^\infty}<+\infty,$$
 then $\theta(s)=K(s, \sigma)g(\sigma)$ satisfies for all $s\in [\sigma,\sigma+\rho]$,
 \begin{enumerate}
 \item \begin{eqnarray*}\displaystyle \|\frac{\theta_-(y,s)}{1+|y|^3}\|_{L^\infty}&\leq &C\frac{e^{s-\sigma}((s-\sigma)^2+1)}{s}(|g_0(\sigma)|+|g_1(\sigma)|+\sqrt{s}|g_2(\sigma)|)\\&&+Ce^{-\frac{s-\sigma}{2}}\|\frac{g_-(y,\sigma)}{1+|y|^3}\|_{L^\infty} +C\frac{e^{-(s-\sigma)^2}}{s^{\frac{3}{2}}}\|g_e(\sigma)\|_{L^\infty}
 \end{eqnarray*}
 \item  $|\theta_e(s)|\leq \displaystyle Ce^{s-\sigma}\big(\sum_{l=0}^2 s^{\frac{l}{2}}|g_l(\sigma)|+s^{\frac{3}{2}}\|\frac{g_{-}(y, \sigma)}{1+|y|^3}\|_{L^\infty}\big)+Ce^{-\frac{s-\sigma}{p}}\|g_e(\sigma)\|_{L^\infty}.$
 \item \begin{eqnarray*}\displaystyle \|(1+|y|^\beta)\theta_e(y,s)\|_{L^\infty}&\leq &\displaystyle C e^{-\frac{s-\sigma}2(\frac 1{p-1}-\frac \beta 2)}\|(1+|y|^\beta)g_e(y,\sigma)\|_{L^\infty}\\
&&+ \displaystyle C  e^{\frac{\beta}{2}(s-\sigma)} s^\frac{\beta}{2}(\sum_{l=0}^2 s^{\frac{l}{2}}|g_l(\sigma)| +s^{\frac{3}{2}}\|\frac{g_-(y,\sigma)}{1+|y|^3}\|_{L^\infty}).
 \end{eqnarray*}
 \end{enumerate}
 \end{lem}
 \begin{proof} Although not written exactly in this form, items $1-2$ were already proved in Bricmont and Kupiainen \cite{bricmont}. In fact, we owe this version to Nguyen and Zaag \cite{NZ2} (see Lemma 3.6 page 1290 in that paper). We therefore give the sketch of the proof only  for part $3)$. We write
 $$(1+|y|^\beta)\theta_e(y,s)=(1+|y|^\beta)(1-\chi(y,s)) K(s, \sigma)(g_e(\sigma)+g_b(\sigma))$$
 For the first term, we remark that
 $$ K(s, \sigma)g_e(\sigma)= \int\frac{K(s,\sigma, y,x)}{1+|x|^\beta}{\bf{1}}_{\{|x|\geq K_0\sqrt{\sigma}\}}(1+|x|^\beta)g_e(x, \sigma)dx$$
Then, using  part $4)$  for Lemma \ref{lem_kernel2}, we obtain
$$\|(1+|y|^\beta)(1-\chi(y,s)) K(s, \sigma)g_e(\sigma)\|_{L^\infty}\leq Ce^{-\frac{s-\sigma}{2}(\frac{1}{p-1}-\frac{\beta}{2})}\|(1+|y|^\beta)g_e(\sigma)\|_{L^\infty}.$$
For the second term, we use a Feynman-Kac representation for $K$:
\begin{equation}\label{feynm}
K(s, \sigma, y, x)=e^{(s-\sigma)\mathcal{L}}(y,x)\int d\nu_{yx}^{s-\sigma}(\omega)\exp \int_0^{s-\sigma}V(\omega(t), \sigma+t)dt,
\end{equation}
where $$\displaystyle e^{(s-\sigma)\mathcal{L}}(y,x)= \frac{e^{s-\sigma}}{\sqrt{4\pi(1-e^{-(s-\sigma)})}} \exp(\frac{-(y e^{-\frac{(s-\sigma)}{2}}-x)^2}{4(1-e^{-(s-\sigma)})}), $$
and $ d\nu_{yx}^{s-\sigma}$  is the oscillator measure on the continuous paths $\omega : [0, s-\sigma] \to    \R$
with $\omega(0) = x$, $ \omega(s -\sigma) = y$, i.e. the Gaussian probability measure with
covariance kernel
\begin{eqnarray*}
\displaystyle
\Gamma(\tau, \tau')&=&\omega_0(\tau) \omega_0(\tau')\\
&+&\displaystyle 2\Big( e^{-\frac{1}{2}|\tau -\tau'|}-e^{-\frac{1}{2}|\tau +\tau'|}+e^{-\frac{1}{2}|2(s-\sigma)+\tau -\tau'|}-e^{-\frac{1}{2}|2(s-\sigma)-\tau -\tau'|}\Big),
\end{eqnarray*}
which yields $\int d\nu_{yx}^{s-\sigma}\omega(\tau)= \omega_0(\tau),$ with 
$$\omega_0(\tau)= (sinh((s-\sigma)/2))^{-1}(y\, sinh(\frac{\tau}{2})+ x \, sinh(\frac{s-\sigma-\tau}{2})\Big).$$
Since $V(y,s) \le \frac C s$ by \eqref{linear}, we readily see that for $\sigma$ large enough, we have  
$$\int d\nu_{yx}^{s-\sigma}(\omega)\exp \int_0^{s-\sigma}V(\omega(t), \sigma+t)dt\leq C. $$
We distinguish two cases:\\
{\bf{Case 1:}} If $x\in \mathbb{A}=\{x; \quad |ye^{-\frac{s-\sigma}{2}}-x|\geq \frac{|x|}{4}\}$, we first write
\begin{equation*}
\exp(\frac{-|ye^{\frac{-(s-\sigma)}{2}}-x|^2}{4(1-e^{-\theta})})=\exp(\frac{-|ye^{\frac{-(s-\sigma)}{2}}-x|^2}{8(1-e^{-(s-\sigma)})})\exp(\frac{-|ye^{\frac{-(s-\sigma)}{2}}-x|^2}{8(1-e^{-(s-\sigma)})}).
\end{equation*}
On the one hand, we remark that since $\displaystyle \frac{|x|}{4}\geq \frac{|ye^{-\frac{(s-\sigma)}{2}}|}{4}-\frac{|ye^{-\frac{(s-\sigma)}{2}}-x|}{4}$, we have $\displaystyle |ye^{-\frac{(s-\sigma)}2}-x|\ge \frac{|x|}4 \ge \frac{|ye^{-\frac{(s-\sigma)}2}|}4 - \frac{|ye^{-\frac{(s-\sigma)}2}-x|}4$, hence
$$|ye^{-\frac{(s-\sigma)}{2}}-x|\geq \frac{|ye^{-\frac{(s-\sigma)}{2}}|}{5}. $$
Thus, \begin{eqnarray*}
\displaystyle \exp(-\frac{|y e^{-\frac{(s-\sigma)}{ 2}} -x|^2}{8(1-e^{-(s-\sigma)})})
& \le&\displaystyle Ce^{-\frac{|ye^{-\frac{(s-\sigma)}{ 2}}|^2}{200(1-e^{-(s-\sigma)})}}
\le \displaystyle Ce^{-\frac{|ye^{-\frac{(s-\sigma)}{ 2}}|^2}{200}} \nonumber \\
&\le&\displaystyle \frac{ C}{1+|ye^{-\frac{ (s-\sigma)}{ 2}}|^\beta} = \frac{Ce^{\frac{\beta(s-\sigma)}{2}}}{e^{\frac{\beta(s-\sigma)}{2}}+|y|^\beta}
\le \frac{Ce^{\frac{\beta(s-\sigma)}{2}}}{1+|y|^\beta}.
\end{eqnarray*}
On the other hand, we write
$$\exp(\frac{-(y e^{-\frac{(s-\sigma)}{2}}-x)^2}{8(1-e^{-(s-\sigma)})})\leq  \exp(-\frac{ |x|^2}{128(1-e^{-(s-\sigma)})}).$$
We conclude that
 $$(1\!+\!|y|^\beta)(1\!-\!\chi(y,s))\!\!\!\int_\mathbb{A}\!\!\! K(s,\sigma, y, x){\bf{1}}_{\{|x|\leq 2K_0\sqrt{\sigma}\}}dx\! \leq\!  C \frac{e^{(1\!+\!\frac\beta 2)(s-\sigma)}}{\sqrt{4\pi (1\!-\!e^{\!-\!(s\!-\!\sigma)})}}\!\!\int_\R\!\! e^{-\frac{|x|^2}{128(1\!-\!e^{\!-\!(s\!-\!\sigma)})}}dx . $$
By change of variables, we obtain
\begin{eqnarray*}
(1+|y|^\beta)(1-\chi(y,s))\int_\mathbb{A} K(s,\sigma, y, x){\bf{1}}_{\{|x|\leq 2 K_0\sqrt{\sigma}\}}dx &\leq & C e^{(1+\frac{\beta}{2})(s-\sigma)}\int_\R e^{-x'^2}dx'\\
&\le& C e^{(1+\frac{\beta}{2})(s-\sigma)}.\end{eqnarray*}
Let  $\rho>0$ and $\sigma$ large enough  such that  $e^{s-\sigma}\le e^\rho\le \sigma^\frac{\beta}{2}\leq s^\frac{\beta}{2}$, hence
$$(1+|y|^\beta)(1-\chi(y,s)) \int_\mathbb{A}K(s,\sigma, y,x){\bf{1}}_{\{|x|\leq 2K_0\sqrt{\sigma}\}}  dx\leq C e^{\frac{\beta}{2}(s-\sigma)} s^\frac{\beta}{2} .$$
{\bf{Case 2:}} If $x\in \mathbb{R}\backslash \mathbb{A}$ where
$\mathbb A$ is defined in Case $1$, we have $\displaystyle |ye^{-\frac{(s-\sigma)}2}|\le |ye^{-\frac{(s-\sigma)}2}-x|+|x|\le \frac{5|x|}4$. Therefore, if $|x|\le 2 K_0 \sqrt \sigma$, then $|y|\le \frac52e^{\frac{(s-\sigma)}2}K_0 \sqrt \sigma$, hence,   we obviously obtain $1+|y|^\beta \leq 1+(\frac{5}{2}K_0)^\beta e^{\frac{\beta}{2}(s-\sigma)} s^\frac{\beta}{2}. $
$$(1+|y|^\beta)(1-\chi(y,s)) \int_{\mathbb{R}\backslash\mathbb{A}}K(s,\sigma, y,x) {\bf{1}}_{\{|x|\leq 2K_0\sqrt{\sigma}\}}  dx\leq  C  e^{\frac{\beta}{2}(s-\sigma)} s^\frac{\beta}{2}. $$
Collecting the above estimates, we obtain for $\sigma$ large enough
$$\|(1+|y|^\beta)(1-\chi(y,s)) K(s,\sigma) g_b(\sigma)\|_{L^\infty}\leq  C   e^{\frac{\beta}{2}(s-\sigma)} s^\frac{\beta}{2}\|g_b(\sigma)\|_{L^\infty}.$$
Using \eqref{decomp}, we write
$$g_b(x,\sigma) = \sum_{l=0}^2 g_l(\sigma)h_l(x) +g_-(x,\sigma). $$
Since $|x|\le 2 K_0 \sqrt{\sigma}$ on the support of $g_b(x,\sigma)$, it follows that
$$\|g_b(\sigma)\|_{L^\infty} \le C\sum_{l=0}^2 s^{\frac{l}{2}}|g_l(\sigma)| +Cs^{\frac{3}{2}}\|\frac{g_-(y,s)}{1+|y|^3}\|_{L^\infty},$$
for large $\sigma$.
Thus, we have 
$$\|(1+|y|^\beta)(1-\chi(y,s)) K(s,\sigma) g_b(\sigma)\|_{L^\infty}\leq  C   e^{\frac{\beta}{2}(s-\sigma)} s^\frac{\beta}{2}\Big(\sum_{l=0}^2 s^{\frac{l}{2}}|g_l(\sigma)| +s^{\frac{3}{2}}\|\frac{g_-(y,s)}{1+|y|^3}\|_{L^\infty}\Big).$$
This concludes the proof of Lemma \ref{lem_14}.
\end{proof}
{\bf{Step 2: Linear estimates on the terms in \eqref{decomp1}.}}\\
Applying Lemma \ref{lem_14}, we get  a new bound  on  all the  terms in the decomposition $(\ref{decomp1})$. More precisely, we have the following:
\begin{lem}\label{lem15} There exists $K_7$ such that for any $K_0\ge
  K_7$, there exists  $A_7>0$ such that for all $A\geq A_7$ and
  $\rho>0$, there exists $s_{07}(K_0, A, \rho)>0$ with the following
  property: for all $s_0\geq s_{07}(K_0, A, \rho)$ assume that for all
  $s\in [\sigma, \sigma+\rho]$, $v(s)$ satisfies $(\ref{eq_v})$,
  $v(s)\in \mathcal{V}_{ A}(s)$ and $\nabla v$ satisfies the estimates
  stated in Proposition  \ref{prop-reg}.
Then, we have:
\begin{enumerate}
\item Linear term:
$$\|\frac{\mathcal{A}_-(y,s)}{1+|y|^3}\|_{L^\infty}\leq Ce^{-\frac{1}{2}(s-\sigma)}\|\frac{v_-(y,\sigma)}{1+|y|^3}\|_{L^\infty}+C\frac{e^{-(s-\sigma)^2}}{s^{\frac{3}{2}}}\|v_e(\sigma)\|_{L^\infty} +\frac{C}{s^2}, $$

$$\|\mathcal{A}_e(s)\|_{L^\infty}\leq Ce^{-\frac{(s-\sigma)}{p}}\|v_e(\sigma)\|_{L^\infty}+ Ce^{s-\sigma}s^{\frac{3}{2}}\|\frac{v_-(y,\sigma)}{1+|y|^3}\|_{L^\infty}+\frac{C}{\sqrt{s}}.$$
 \begin{eqnarray*}
\|(1+|y|^\beta)\mathcal{A}_e(s)\|_{L^\infty}&\leq& C e^{-\frac {(s-\sigma)}2 (\frac 1{p-1}-\frac\beta2)}\| (1+|y|^\beta)v_e(\sigma)\|_{L^\infty}+ C  e^{\frac{\beta}{2}(s-\sigma)} s^{\frac{\beta}{2}+\frac{3}{2}}\|\frac{v_-(y,s)}{1+|y|^3}\|_{L^\infty}\\&&+\frac{C}{s^{\frac{1-\beta}{2}}}.
 \end{eqnarray*}
 \item Nonlinear source term:
 $$\|\frac{\mathcal{B}_-(y,s)}{1+|y|^3}\|_{L^\infty}\leq \frac{C}{s^2}(s-\sigma), \quad  \|\mathcal{B}_e(s)\|_{L^\infty}\leq \frac{C}{\sqrt{s}}(s-\sigma)e^{s-\sigma},$$
 $$\|(1+|y|^\beta)\mathcal{B}_e(s)\|_{L^\infty}\leq C\frac{1+ e^{\frac{\beta}{2}(s-\sigma)} (s-\sigma)}{s^{\frac{1-\beta}{2}}}.$$
 \item Corrective term:
 $$\|\frac{\mathcal{C}_-(y,s)}{1+|y|^3}\|_{L^\infty}\leq \frac{C}{s^2}(s-\sigma), \quad  \|\mathcal{C}_e(s)\|_{L^\infty}\leq \frac{C}{\sqrt{s}}(s-\sigma)e^{s-\sigma},$$
 $$\|(1+|y|^\beta)\mathcal{C}_e(s)\|_{L^\infty}\leq C\frac{1+ e^{\frac{\beta}{2}(s-\sigma)} (s-\sigma)}{s^{\frac{1-\beta}{2}}}.$$
 \item New term:
 $$\|\frac{\mathcal{D}_-(y,s)}{1+|y|^3}\|_{L^\infty}\leq C(s-\sigma)e^{s-\sigma}e^{-\frac{\gamma}{2}s}, \quad  \|\mathcal{D}_e(s)\|_{L^\infty}\leq C(s-\sigma)e^{s-\sigma}e^{-\frac{\gamma}{2}s},$$
 $$\|(1+|y|^\beta)\mathcal{D}_e(s)\|_{L^\infty}\leq C(s-\sigma)e^{-\frac{\gamma}{2}s}(1+ e^{\frac{\beta}{2}(s-\sigma)} s^\frac{\beta}{2}).$$
\end{enumerate}
\end{lem}
\begin{proof} 
Since for all $s\in[\sigma, \sigma+\rho]$, $v(s) \in {\cal V}_A(s)$ and $\nabla v(s)$ satisfies the estimates in Proposition \ref{prop-reg}, using Lemma  
\ref{lem_est} and  Proposition \ref{prop_new_term}, we see that $v(\sigma)\in {\cal V}_A(\sigma)$, $B(v)\in {\cal V}_C(s)$, $R(s)\in  {\cal V}_C(s)$ and $N(s)\in V_C(s)$. Therefore, we can apply the linear estimates of Lemma \ref{lem_14} to obtain the desired
estimates (after an integration in time, if necessary). One may see Lemma 4.20 page 5940 in \cite{TZ} for a similar argument.\\
In fact, for the reader's convenience, we will present the proof for
$\mathcal D(s)= \int_\sigma^sK(s, t) N(t)dt$, since this is the new
term in our study and since it shows an exponentially decaying factor
which cannot be obtained from a mechanical application of the
above-mentioned lemmas.\\
Using Lemma \ref{lem_kernel2} and Lemma \ref{lem_new_term}, we deduce that
$$|\mathcal D(s)|\leq C(s-\sigma)e^{s-\sigma}e^{-\frac{\gamma}{2}s}.$$
In particular, for $0\leq i \leq 2$, arguing as for $N_i(s)$ in the proof of Proposition \ref{prop_new_term} page \pageref{proof3.10}, we get 
$$|\mathcal D_i(s)|\leq C(s-\sigma)e^{s-\sigma}e^{-\frac{\gamma}{2}s}.$$
Then, we write
\begin{eqnarray*}
|\mathcal D_-(s)|&=&|\mathcal D_b(s)-\sum_{i=0}^2 \mathcal D_i(s) k_i(y)|\\
&\leq& C(s-\sigma)e^{s-\sigma}e^{-\frac{\gamma}{2}s}(1+|y|+|y|^2) \\
&\leq& C(s-\sigma)e^{s-\sigma}e^{-\frac{\gamma}{2}s}(1+|y|^3).
\end{eqnarray*}
Also, we see that
$$|\mathcal D_e(s)|=|(1-\chi(y, s)) \mathcal D(s)|\leq |\mathcal D(s)|\leq C(s-\sigma)e^{s-\sigma}e^{-\frac{\gamma}{2}s}. $$
Finally, using Lemma \ref{lem_14} Part $3)$,  Proposition \ref{prop_new_term} and  Lemma \ref{lem_new_term}, we write 
\begin{eqnarray*}
\|(1+|y|^\beta)\mathcal D_e(s)\|_{L^\infty}&\leq& \int_\sigma^s \|(1+|y|^\beta)(1-\chi(y,s))K(s, t, y, x) N(x)\|_{L^\infty} dt\\
&\leq & C  \int_\sigma^s e^{-\frac{s-t}2(\frac 1{p-1}-\frac \beta 2)}\|(1+|y|^\beta)N_e(y)\|_{L^\infty}  dt\\
&&+  C \displaystyle  \int_\sigma^s  e^{\frac{\beta}{2}(s-t)} s^\frac{\beta}{2} (\sum_{l=0}^2 s^{\frac{l}{2}}\|N_l(t)\|_{L^\infty}+s^{\frac{3}{2}} \|\frac{N_-(y)}{1+|y|^3}\|_{L^\infty})  dt.
\end{eqnarray*}
This yields part $4$ and concludes the proof of Lemma \ref{lem15}.
\end{proof}
{\bf{Step 3 : Conclusion of the proof of Proposition \ref{pro_red}.}}\\ Thanks to the integral equation \eqref{decomp1}, we clearly see
 that items iii) and iv) follow by adding all the estimates in Lemma \ref{lem15}. It remains then to justify items i) and ii),
which we give in the following.\\
Without the perturbation term in \eqref{eq_u}, i.e. when $\mu=0$, the proof can be found in any paper about the standard 
semilinear heat equation, for example in Lemma 3.10 page 1293 of Nguyen and Zaag \cite{NZ2}.\\
With the perturbation term, i.e. when $\mu \neq 0$, we see from Proposition \ref{prop_new_term} that the contribution of the new term 
is exponentially small, hence, under the error terms in the ODEs shown in i)-ii). 
This  concludes the proof of Proposition \ref{pro_red}.
 \end{proof}
 Let us now give the proof of  Proposition \ref{prop_red}.
\begin{proof}[\textbf{Proof of  Proposition \ref{prop_red}}]
  We sketch only the proof of part i), since part $ii)$ follows from item i) of Proposition \ref{pro_red} exactly as in \cite{TZ} (see item ii) of Proposition 4.6 page $5914$ in that paper).\\ Let $v$ be a solution of equation $(\ref{eq_v})$ with initial data $\psi_{s_0, d_0, d_1}$ given by $(\ref{ci})$ with $(d_0, d_1)\in \mathcal{D}_{s_0}$ defined in Proposition \ref{prop_ci}, such that $v(s)\in \mathcal{V}_{ A}(s)$ for all $s\in [s_0, s_1]$ with $v(s_1)\in \partial\mathcal{V}_{ A}(s_1)$. By Definition \ref{def_shrinking} of the shrinking set, in order to conclude, it is enough   to prove that
 \begin{eqnarray}\label{3.61}
 |v_2(s_1)|<\displaystyle \frac{A^2 \log s_1}{s_1^2}, \quad \|\frac{v_-(y,s_1)}{1+|y|^3}\|_{L^\infty}< \frac{A}{s_1^2}\\
 \|v_e(s_1)\|_{L^\infty}< \frac{A^2}{\sqrt{s_1}}, \quad \| (1+|y|^\beta)v_e(s_1)\|_{L^\infty}<\frac{A^2}{s_1^{\frac{1-\beta}{2}}}.\nonumber
 \end{eqnarray}
 We prove only the last inequality, and refer  the reader to \cite{TZ} (precisely to item i) of Proposition 4.6 there) for the proof of the other estimates.\\
We will in fact choose $\rho$ of the form  $\rho=\log(A^\alpha)$, $\alpha<1$. Then, taking $s_0$ large enough, we may have the following
 \begin{equation}\label{rho}
 \rho \leq s_0<\sigma<s=\sigma +\rho<2\sigma.
 \end{equation}
 Let us first assert that the parabolic regularity result stated in Proposition \ref{prop-reg}  holds. In particular, we can apply all the statements requiring its conclusion, in particular Proposition \ref{pro_red}.\\
 We distinguish two cases:\\
 {\bf{Case 1:}} $s>s_0+\rho$. Hence, $\sigma=s-\rho>s_0$ and from
 Proposition \ref{pro_red} part $iv)$ together with the fact
that $v(\sigma)
\in {\cal V}_A
(\sigma)$, we have
 \begin{eqnarray*}
  \| (1+|y|^\beta)v_e(s)\|_{L^\infty}&\leq& C e^{-\frac {(s-\sigma)}2(\frac 1{p-1} - \frac \beta 2)} \| (1+|y|^\beta)v_e(\sigma)\|_{L^\infty}+C e^{\frac{\beta}{2}(s-\sigma)} s^{\frac{\beta}{2}+\frac{3}{2}}\|\frac{v_-(y,\sigma)}{1+|y|^3}\|_{L^\infty}\\
  &&+C\frac{1+(s-\sigma) e^{\frac{\beta}{2}(s-\sigma)} }{s^{\frac{1-\beta}{2}}}\\
  &\leq &  \displaystyle  C\frac{A^2 e^{-\frac {\rho}2(\frac 1{p-1} - \frac \beta 2)}+Ae^{\frac{\rho \beta}{2}}+1+\rho e^{\frac{\rho\beta }{2}}}{s^{\frac{1-\beta}{2}}}\\
  &\leq &   C\frac{A^{2-\frac {\alpha}2(\frac 1{p-1} - \frac \beta 2)}+A^{1+\frac{\alpha \beta}{2}}+1+\log(A^\alpha)A^{\frac{\alpha\beta }{2}}}{s^{\frac{1-\beta}{2}}}.
 \end{eqnarray*}
 Since $\alpha<1$ and $\beta<\frac{2}{p-1}<1$ by \eqref{hyp} and \eqref{beta},  taking $A$ sufficiently large, we see that
 $$\| (1+|y|^\beta)v_e(s)\|_{L^\infty}\leq \frac{1}{2}\frac{A^2}{s^{\frac{1-\beta}{2}}}. $$
 As we said earlier, the same argument allows to show the other estimates in $(\ref{3.61})$ with the same choice of
$\rho = \log(A^\alpha)$ (see item i) of Proposition 4.6 page $5914$ in \cite{TZ}).\\
 {\bf{Case 2:}} $s<s_0+\rho$. Clearly, from this choice, we have $s_0<s<s_0+\rho<2s_0$, by \eqref{rho}.
 If we choose $\sigma=s_0$, then $v_e(s_0)=\psi_e(s_0)=0$ by item $ii)$ of Proposition \ref{prop_ci}. Using  Proposition \ref{pro_red} part $iv)$, we have
   $$\| (1+|y|^\beta)v_e(s)\|_{L^\infty}\leq C\frac{ e^{\frac{\rho \beta}{2}}+1+\rho e^{\frac{\rho \beta}{2}} }{s^{\frac{1-\beta}{2}}}. $$
Taking $A$ sufficiently large, such that $$C(A^{\frac{\alpha\beta}{2}}+1+\log(A^\alpha)A^{\frac{\alpha\beta}{2}})\leq \frac{A^2}{2},$$
 we conclude that for all $s\in [s_0, s_1]$
   $$\| (1+|y|^\beta)v_e(s)\|_{L^\infty}\leq \frac{1}{2}\frac{A^2}{s^{\frac{1-\beta}{2}}}. $$
Again, the other estimates in \eqref{3.61} hold also with the same proof (see item i) of Proposition  4.6 page $5914$ in \cite{TZ}).\\
   In particular, in  the two cases, we have
   $$\| (1+|y|^\beta)v_e(s_1)\|_{L^\infty}<\frac{A^2}{s_1^{\frac{1-\beta}{2}}}, $$
   and \eqref{rho} holds.\\
 Since $v(s_1)\in \partial\mathcal{V}_{ A}(s_1)$, we see that $(v_0(s_1), v_1(s_1))\in \displaystyle \partial [-\frac{A}{s_1^2}, \frac{A}{s_1^2}]^2$ and part $i)$ of Proposition \ref{prop_red} is proved. Since item ii) follows by the same argument as item ii) of Proposition  4.6 page $5914$ in \cite{TZ}, this concludes the proof of Proposition \ref{prop_red}.
\end{proof}
 \section{Proof of Theorem \ref{th1} and Corollary \ref{corol}}
  In this section, we prove our main result, using the previous subsections. We recall that, from Proposition \ref{prop-reg} and Proposition \ref{exis}, we obtain the existence of a solution $v$ of equation $(\ref{eq_v})$ defined for all $y\in \R$ and $s\geq s_0$, for some $s_0\ge 1$ such that $v(s)\in \mathcal{V}_{A}(s)$. More precisely, we have
  $$\|(1+|y|^\beta)v(s)\|_{L^\infty}+\|(1+|y|^\beta)\nabla v(s)\|_{L^\infty}\leq \displaystyle \frac{CA^2}{s^{\frac{1-\beta}{2}}}.$$
 Thus, by definition \eqref{def_v} of $v$, we have for all $s\ge s_0$,
 \begin{equation}\label{est}
\|( 1+|y|^\beta)(w(y,s)-\varphi(y,s))\|_{L^\infty}+\|(1+|y|^\beta)\nabla (w(y,s)-\varphi(y,s))\|_{L^\infty}\leq \displaystyle \frac{C}{s^{\frac{1-\beta}{2}}},
 \end{equation}
where $\varphi$ is the profile introduced in $(\ref{n_profile})$. Let us first estimate this profile. We give the following Lemma:
 \begin{lem}\label{Lem_profile}
 \begin{enumerate}
 \item For all $K_0>0$ and $y$ such that $|y|\geq 2K_0\sqrt{s}$, we have
 \begin{itemize}
 \item[i)]  $|\displaystyle f(\frac{y}{\sqrt{s}})|\sim (\frac{s}{b|y|^{2}})^{\frac{1}{p-1}}$ as $\displaystyle \frac{y}{\sqrt{s}}\to +\infty$.
 \item[ii)] $\displaystyle |\frac{1}{\sqrt{s}}f'(\frac{y}{\sqrt{s}})|\leq \frac{C}{(1+|y|^\beta)s^{\frac{1-\beta}{2}}}$.
 \end{itemize}
\item For all $y\in \R$,
\begin{itemize}
\item[i)] $\displaystyle|\frac{\kappa}{2ps}\chi_0(\frac{y}{g_\eps(s)})|\leq \frac{C}{(1+|y|^\beta)s^{\frac{1-\beta}{2}}}$.
\item[ii)] $\displaystyle|\frac{\kappa}{2ps}\frac{1}{g_\eps(s)}\chi'_0(\frac{y}{g_\eps(s)})|\leq \frac{C}{(1+|y|^\beta)s^{\frac{1-\beta}{2}}}$.
\end{itemize}
 \end{enumerate}
 \end{lem}
 \begin{proof}
 For the proof of $1)$, we recall from \eqref{profile} that $f(z)=(p-1+bz^2)^{-\frac{1}{p-1}}$, where $z=\displaystyle \frac{y}{\sqrt{s}}$. Note that item $i)$ follows from  a Taylor expansion.\\
 We now present the proof of item $ii)$. We write
 \begin{eqnarray*}\frac{1}{\sqrt{s}}f'(\frac{y}{\sqrt{s}})&=&-\frac{2b}{p-1}\frac{y}{s}(p-1+b\frac{y^2}{s})^{-\frac{p}{p-1}}\\
 &=&-\frac{2b^{\frac{1}{p-1}}}{p-1}\frac{1}{\sqrt{s}}z^{-\frac{p+1}{p-1}}(1+\frac{p-1}{bz^2})^{-\frac{p}{p-1}}.
 \end{eqnarray*}
 Since $|z|\geq 2K_0$ and $\beta<\frac{2}{p-1}$ by \eqref{beta}, we have
 $$|(1+|y|^\beta)\frac{1}{\sqrt{s}}f'(\frac{y}{\sqrt{s}})|\leq C\frac{|z|^{\beta-\frac{p+1}{p-1}}}{s^{\frac{1-\beta}{2}}}.$$
 Thus, $ii)$ is proved.\\
 For the proof of $2)$, we see that
 $$|(1+|y|^\beta)\frac{1}{s}\chi_0(\frac{y}{g_\eps(s)})|\leq \frac{C}{s}(1+|y|^\beta){\bf{1}}_{\{|y|\leq 2K_0 g_\eps(s)\}}\leq \frac{C}{s} |g_\eps(s)|^\beta=\frac{C}{s^{1-\frac{\beta}{2}-\beta\eps}}. $$
 Since $\eps<\frac{1}{2\beta}$ by \eqref{beta} and the line right before \eqref{def_v}, we get the conclusion of $2)$ $i)$. The  proof of $ii)$ follows exactly as above.
 This concludes the proof of Lemma \ref{Lem_profile}.
 \end{proof}

 \bigskip

 Now, we give the proof of Theorem \ref{th1}.


 \begin{proof}[\textbf{Proof of Theorem \ref{th1}}]
 Using the above Lemma and inequality  $(\ref{est})$, we deduce that
\begin{equation*}
\|(1+|y|^\beta)(w(y,s)-f(\frac{y}{\sqrt{s}}))\|_{L^\infty}+\|(1+|y|^\beta)\nabla_y (w(y,s)-f(\frac{y}{\sqrt{s}}))\|_{L^\infty}\leq \displaystyle \frac{C}{s^{\frac{1-\beta}{2}}}.
\end{equation*}
In particular, by the similarity variables transformation \eqref{var_sim},  the solution $u$ of equation $(\ref{eq_u})$ defined for all $x\in \R$ and $t\in [0, T)$ satisfies
\begin{equation*}
|(T-t)^{\frac{1}{p-1}}u(x,t)-f(\frac{x}{\sqrt{(T-t)|\log(T-t)|}})|\leq \displaystyle \frac{C}{(1+(\frac{|x|^2}{T-t})^\frac{\beta}{2})|\log(T-t)|^{\frac{1-\beta}{2}}},
\end{equation*}
hence $\lim_{t\to T} (T-t)^{\frac{1}{p-1}}u(0, t)=(p-1)^\frac{1}{p-1}$ and $u$ blows up at time $t=T$ at the origin.\\
On the other hand, we remark that
\begin{eqnarray*}
|(T-t)^{\frac{1}{2}+\frac{1}{p-1}}\nabla u(x,t)&-&\frac{1}{\sqrt{|\log(T-t)|}} \nabla f(\frac{x}{\sqrt{(T-t)|\log(T-t)|}})|\\
&&\leq \displaystyle \frac{C}{(1+(\frac{|x|^2}{T-t})^\frac{\beta}{2})|\log(T-t)|^{\frac{1-\beta}{2}}}.
\end{eqnarray*}
This concludes the proof of Theorem \ref{th1}.
\end{proof}
Now,  let us briefly explain the  proof of Corollary  \ref{corol}.
\begin{proof}[\textbf{Proof of Corollary  \ref{corol}}]
From Theorem \ref{th1} and Lemma \ref{Lem_profile}, we obtain 
$$\displaystyle u(x,t)= (\frac{|\log(T-t)|}{b|x|^2})^{\frac{1}{p-1}}(1+O(1))+O\Big((\frac{|\log(T-t)|}{b|x|^2})^{\frac{\beta}{2}}\frac{(T-t)^{\frac{\beta}{2}-\frac{1}{p-1}}}{ |\log(T-t)|^{\frac{1}{2}}}\Big),$$
and 
$$\displaystyle \nabla u(x,t)= c(\frac{|\log(T-t)|}{|x|^{p+1}})^{\frac{1}{p-1}}(1+O(1))+O\Big( (\frac{|\log(T-t)|}{b|x|^2})^{\frac{\beta}{2}}\frac{(T-t)^{\frac{\beta}{2}-\frac{1}{p-1}-\frac{1}{2}}}{|\log(T-t)|^{\frac{1}{2}}}\Big),$$
for some $c\neq 0$, as $\displaystyle \frac{|x|}{\sqrt{(T-t)|\log(T-t)|}} \to \infty$, uniformly in $t\in [0,T)$.\\
Since $\displaystyle \frac{|\log(T-t)|}{|x|^2}< \frac{1}{ (T-t) |\log(T-t)|^{(\frac{2}{p-1}-\beta)^{-1}}}$ and $\beta<\frac{2}{p-1}$ by \eqref{beta}, we conclude the proof of Corollary \ref{corol}.
\end{proof}

\appendix

\section{Regularizing effect of the operator $\cal L$ in the weighted space}
In this appendix, we give the proof of Lemma \ref{lem_kernel1}. As mentioned earlier, we will give the proof  of part $5)$ only, since the other estimates are  proved in  Lemma $4$ page $555$ of  \cite{bricmont} and Lemma      $4.15$ page $5926 $ of \cite{TZ}.\\
By definition $(\ref{kernel_e})$, we write
\begin{eqnarray}\label{A1}
(e^{\theta \mathcal{L}}r)(y)&=&\displaystyle \int_\R \frac{e^\theta}{\sqrt{4\pi (1-e^{-\theta})}}e^{-\frac{(ye^{-\frac{\theta}{2}}-x)^2}{4 (1-e^{-\theta})}} r(x) dx\\
&=&  \int_\R \frac{e^\theta}{\sqrt{4\pi (1-e^{-\theta})}}e^{\frac{-z^2}{4 (1-e^{-\theta})}} r(ye^{-\frac{\theta}{2}}-z) dz,\nonumber
\end{eqnarray}
where $z=ye^{-\frac{\theta}{2}}-x$, and
\begin{eqnarray*}
\nabla (e^{\theta \mathcal{L}}r)(y)
&=& \displaystyle \int_\R \frac{e^{\frac{\theta}{2}}}{\sqrt{4\pi (1-e^{-\theta})}}e^{\frac{-z^2}{4 (1-e^{-\theta})}} \nabla r(ye^{-\frac{\theta}{2}}-z) dz.
\end{eqnarray*}
Since 
\begin{equation}\label{A2}
|y|^m\leq C e^{\frac{m\theta}{2}}(|ye^{-\frac{\theta}{2}}-z|^m+|z|^m),
\end{equation}
we have
\begin{eqnarray*}
|(1+|y|^m)\nabla (e^{\theta \mathcal{L}}r)(y)|&\leq& C e^{\frac{(m+1)\theta}{2}}\Big(\int_\R (1+|ye^{-\frac{\theta}{2}}-z|^m)
\frac{ e^{\frac{-z^2}{4 (1-e^{-\theta})}}}{\sqrt{4\pi (1-e^{-\theta})}}\nabla r(ye^{-\frac{\theta}{2}}-z)  dz\\
&+&\int_\R |z|^m  \frac{e^{\frac{-z^2}{4 (1-e^{-\theta})}}}{\sqrt{4\pi (1-e^{-\theta})}}\nabla r(ye^{-\frac{\theta}{2}}-z)  dz\Big). \\
&\leq& C e^{\frac{(m+1)\theta}{2}}\Big(\|(1+|y|^m)\nabla r(y)\|_{L^\infty}\int_\R
\frac{ e^{\frac{-z^2}{4 (1-e^{-\theta})}}}{\sqrt{4\pi (1-e^{-\theta})}}  dz\\
&&+ \|\nabla r(y)\|_{L^\infty}\int_\R
\frac{ |z|^m e^{\frac{-z^2}{4 (1-e^{-\theta})}}}{\sqrt{4\pi (1-e^{-\theta})}}  dz\Big).
\end{eqnarray*}
We remark that, for $\alpha \geq 0$, we get
$$\int_\R
\frac{ |z|^\alpha e^{\frac{-z^2}{4 (1-e^{-\theta})}}}{\sqrt{4\pi (1-e^{-\theta})}}  dz\leq C (1-e^{-\theta})^{\frac{\alpha}{2}}\leq C. $$
This yields  part $i)$ of $5)$.\\ In order to prove the part $ii)$, we write from \eqref{A1}
$$\nabla(e^{\theta \mathcal{L}}r)(y)=\displaystyle \int_\R \frac{e^\theta}{\sqrt{4\pi (1-e^{-\theta})}}\frac{-e^{-\frac{\theta}{2}}(ye^{-\frac{\theta}{2}}-x)}{2(1-e^{-\theta})} e^{-\frac{(ye^{-\frac{\theta}{2}}-x)^2}{4 (1-e^{-\theta})}} r(x) dx.$$
If we make the change of variable $z=ye^{-\frac{\theta}{2}}-x$, we obtain
$$\nabla(e^{\theta \mathcal{L}}r)(y)=\displaystyle 2\int_\R \frac{e^{\frac{\theta}{2}}}{\sqrt{\pi} (4(1-e^{-\theta}))^\frac{3}{2}} ze^{-\frac{z^2}{4(1-e^{-\theta})}}  r(ye^{-\frac{\theta}{2}}-z ) dz.$$
In particular, using \eqref{A2}, we get 
\begin{eqnarray*}
|(1+|y|^m)\nabla(e^{\theta \mathcal{L}}r)(y)|\!\!\!&\leq&\!\!\! \!\! c\!\int_\R\! \frac{e^{\frac{(m+1)\theta}{2}}}{ (1\!-\!e^{-\!\theta})^\frac{3}{2}} |z|e^{\!-\!\frac{z^2}{4(1-e^{-\theta})}}(1+|ye^{-\!\frac{\theta}{2}}\!-\!z|^m)  r(ye^{-\frac{\theta}{2}\!}-\!z ) dz\\
\!\!\!&+& \!\!\! c\int_\R \frac{e^{\frac{(m+1)\theta}{2}}}{(1-e^{-\theta})^\frac{3}{2}} |z|^{m+1} e^{-\frac{z^2}{4(1-e^{-\theta})}} r(ye^{-\frac{\theta}{2}}-z ) dz.\\
\end{eqnarray*}
Using the fact that for all $\alpha\geq 0$, $\int_\R|\frac{z}{\sqrt{1-e^{-\theta}}}|^\alpha e^{-\frac{z^2}{4(1-e^{-\theta})}}dz\leq C$, we obtain
$$\|(1+|y|^m)\nabla(e^{\theta \mathcal{L}}r)(y)\|_{L^\infty} \leq \displaystyle C( \frac{e^{\frac{(m+1)\theta}{2}}}{\sqrt{1-e^{-\theta}}} \|(1+|y|^m)r(y)\|_{L^\infty}+\frac{e^{\frac{(m+1)\theta}{2}}(1-e^{-\theta})^{\frac{m}{2}}}{\sqrt{1-e^{-\theta}}} \|r(y)\|_{L^\infty}.$$
Thus, part $ii)$ of $5)$ is proved.\\

\section{Effect at infinity of the operator ${\cal L}+V$ in the weighted space}
We prove now  Lemma \ref{lem_kernel2}.
It happens that the operator $\mathcal{L}+V$ is a little different from the operator in \cite{bricmont}, in the sense that we have a  cut-off in our profile $\varphi$ defined in \eqref{n_profile} (just take $\chi_0\equiv 1$ in \eqref{n_profile} to exactly recover the operator in \cite{bricmont}). However, despite this difference,  the spectral analysis  in \cite{bricmont} holds here and one gets the proof  of
 estimates $1)$ and $ 2) $  similarly (see Lemma $4$ page $555$  in \cite{bricmont}
 together with the Feynman-Kac formula given in \eqref{feynm}). Thus, we only prove estimate $3)$ and $4)$.\\
 We would like to mention that  estimate 3) was proved  in \cite{bricmont} with a bound involving  $e^{-\frac{s-\sigma}{p}}$,  which means that our estimate in $3)$ is already proved  when $[\eta =\frac{1}{p-1}-\frac{1}{p}]$. It happens in fact that the proof of \cite{bricmont} holds for any $\eta>0$. Indeed, estimate $3)$ is in fact the formal translation of the fact that the linear operator $\mathcal{L}+V$ behaves like $\mathcal{L}-\displaystyle \frac{p}{p-1}Id$ as $s$ goes to infinity and $\frac{|y|}{\sqrt{s}}$ goes to  infinity, as we have already mentioned in Section $2$, and as we will recall here.\\
 Indeed, we have by definition of potential $V(y, s)$ in \eqref{linear} that
 $$ V(y, s) \to-\displaystyle \frac{p}{p-1} , \; \mbox{as}\; s\to \infty; \, \frac{|y|}{\sqrt{s}}\to \infty.$$
 Since we know  that the highest eigenvalue of $\mathcal{L}$ is $1$ (see \eqref{spect}), it is reasonable to think that  the operator   $\mathcal{L}+V$ has  its largest eigenvalue bounded  by
 $$1-\frac{p}{p-1}+\eta=-\frac{1}{p-1}+\eta, $$
where $\eta>0$ is arbitrarily small, which fully justifies our estimate $3)$.\\
For the proof, see Lemma $7$ page $559$ in \cite{bricmont}, which holds even with this more general version.\\

Let us now give the proof of $4)$. Consider $m\in [0,\frac 2{p-1})$ and introduce
\begin{equation}\label{defeta}
\eta=\frac 12\left( \frac 1{p-1}-\frac m2\right)>0.
\end{equation}
Consider also $K_0>0$, $\rho>0$ and $\sigma\ge 1$ to be taken large enough, and $s\in[\sigma, \sigma+\rho]$.\\
We recall from \eqref{feynm} the Feynman-Kac representation for $K$:

\begin{equation}\label{Kernel}
K(s,\sigma, y, x)= e^{\theta \mathcal{L}}(y,x) \int d\nu^\theta_{yx} \exp(\int_0^\theta V(w(t), \sigma +t) dt,
\end{equation}
where   $\theta=s-\sigma$ and  $e^{\theta \mathcal{L}}(y,x)= \frac{e^\theta}{\sqrt{4\pi (1-e^{-\theta})}}\exp(\frac{-|ye^{\frac{-\theta}{2}}-x|^2}{4(1-e^{-\theta})})$ and the meaure $d\nu_{yx}^\theta$ was 
introduced after \eqref{feynm}.\\
We distinguish two cases:\\
{\bf{Case 1:}} If $x\in \mathbb{A}=\{x,|ye^{-\frac{\theta}{2}}-x|\geq \frac{|x|}{4}  \}$.\\
We first remark that since $\displaystyle \frac{|x|}4 \ge \frac{|y e^{-\frac \theta 2}|}4 - \frac{|ye^{-\frac \theta 2} -x|}4$, we see that $$\displaystyle |ye^{-\frac\theta 2}-x|\ge \frac{|x|}4 \ge  \frac{|ye^{-\frac\theta 2}|}4 -\frac{|ye^{-\frac \theta 2}-x|}4,$$ hence 
\begin{equation}\label{new}
\displaystyle
|ye^{-\frac \theta 2} -x|\ge \frac{|ye^{-\frac \theta 2}|}5.
\end{equation}
Arguing as in the proof of Lemma \ref{lem_14}, we prove that
\begin{equation}\label{prob}
\int d\nu^\theta_{yx} \exp(\int_0^\theta V(w(t), \sigma +t) dt\leq C,
\end{equation}
provided that $\sigma\ge s_3(\rho)$ for large enough $s_3(\rho)$. Then, we decompose
\begin{equation}\label{dec_exp}
\exp(\frac{-|ye^{\frac{-\theta}{2}}-x|^2}{4(1-e^{-\theta})})=\exp(\frac{-|ye^{\frac{-\theta}{2}}-x|^2}{12(1-e^{-\theta})})\exp(\frac{-|ye^{\frac{-\theta}{2}}-x|^2}{12(1-e^{-\theta})})\exp(\frac{-|ye^{\frac{-\theta}{2}}-x|^2}{12(1-e^{-\theta})}).
\end{equation}
First, we remark  form the hypothesis of this case that when $|x|\geq K_0\sqrt{\sigma}$
\begin{equation}\label{B0}
\exp(\frac{-|ye^{\frac{-\theta}{2}}-x|^2}{12(1-e^{-\theta})})\leq C e^{-\frac{1}{192}\frac{|x|^2}{1-e^{-\theta}}}\leq C e^{-\frac{|x|^2}{192}}\leq Ce^{-\frac{K_0^2}{192}\sigma}  .\end{equation}
Second, using \eqref{new}, we write
\begin{eqnarray}
\displaystyle \exp(-\frac{|ye^{-\frac \theta 2} -x|^2}{12(1-e^{-\theta})})
&\displaystyle \le &Ce^{-\frac{|ye^{-\frac \theta 2}|^2}{300(1-e^{-\theta})}}
\le  Ce^{-\frac{|ye^{-\frac \theta 2}|^2}{300}} \nonumber \\
&\displaystyle \le& \frac{ C}{1+|ye^{-\frac \theta 2}|^m} = \frac{Ce^{\frac{m\theta}{2}}}{e^{\frac{m\theta}{2}}+|y|^m}
\le \frac{Ce^{\frac{m\theta}{2}}}{1+|y|^m}.\label{(4.60)}
\end{eqnarray}
Using  \eqref{Kernel}, \eqref{prob}, \eqref{dec_exp}, \eqref{B0} and \eqref{(4.60)}, we get
$$\int_\mathbb{A}\frac{K(s,\sigma, y, x)}{1+|x|^m}{\bf{1}}_{\{|x|\geq K_0\sqrt{\sigma}\}}dx \leq  C \frac{e^{\theta}e^{-\frac{K_0^2}{300}\sigma}}{\sqrt{4\pi (1-e^{-\theta})}}\frac{e^{\frac{m\theta}2}}{1+|y|^m}\int_\R e^{-\frac{|x|^2}{192(1-e^{-\theta})}}dx . $$
By change of variables , we obtain
$$\int_\mathbb{A}\frac{K(s,\sigma, y, x)}{1+|x|^m}{\bf{1}}_{\{|x|\geq K_0\sqrt{\sigma}\}}dx \leq  C \frac{e^{(1+\frac m 2)\theta-\frac{K_0^2}{300} \sigma}}{1+|y|^m}\int_\R e^{-x'^2}dx'\le C \frac{e^{(1+\frac m 2)\theta-\frac{K_0^2}{300} \sigma}}{1+|y|^m} . $$
We claim that $(1+\frac m 2)\theta-\frac{K_0^2}{300} \sigma\leq-\eta \theta$, for $\sigma $ large enough, where $\eta$ is introduced in \eqref{defeta}.
 Indeed, if $\sigma\le s_4(K_0,\rho)$ for some $s_4(K_0, \rho)$, we write for $s\in[\sigma, \sigma+\rho]$ and $\theta=s-\sigma$, 
$$(1+\frac m2 +\eta)\theta \le \frac{K_0^2}{300}\sigma$$ and the identity follows.
Thus, we conclude that
$$\int_\mathbb{A}\frac{K(s,\sigma, y, x)}{1+|x|^m}{\bf{1}}_{\{|x|\geq K_0\sqrt{\sigma}\}}dx \leq  C \frac{e^{-\eta \theta }}{1+|y|^m}. $$
{\bf{Case 2:}} If $x\in \R\setminus \mathbb{A}$,  we remark that $|ye^{-\frac{\theta}{2}}|\in [\frac{3}{4}|x|, \frac{5}{4}|x|]$. Therefore,
$$\displaystyle \frac{1}{1+|x|^m}\leq \frac{(\frac{5}{4}e^{\frac{\theta}{2}})^m }{1+|y|^m}.$$
Using part $2)$ of Lemma \ref{lem_kernel2}, we obtain there exist $\eta$ such that
$$\int_{\R \setminus\mathbb{A}}\frac{K(s,\sigma, y, x)}{1+|x|^m}{\bf{1}}_{\{|x|\geq K_0\sqrt{\sigma}\}}dx \leq  C \frac{e^{-(\frac{1}{p-1}-\eta-\frac{m}{2})\theta}}{1+|y|^m}= C \frac{e^{-\eta\theta}}{1+|y|^m},  $$
by definition \eqref{defeta} of $\eta$.
This concludes the proof of  estimate $4)$   of Lemma \ref{lem_kernel2}.\\
\section{Cauchy problem in the weighted space}
In this appendix, we prove the existence and uniqueness of a solution of equation $(\ref{eq_u})$ in the functional space $W^{1, \infty}_\beta(\R^N)$ by a fixed point argument, provided that 
$$ p \ge 1,\,  \beta(q-1)>N,\,  q\ge 2.$$
Note that these conditions are more general than the conditions in our present paper, namely \eqref{hyp} and \eqref{beta}.\\
Let $S(t)$ be the heat semigroup and let us write the equation  $(\ref{eq_u})$ in its Duhamel formulation:
$$u(t)=S(t)u_0+\int_0^t S(t-s)(g(u)+h(u))ds,$$
where $g(u)=|u|^{p-1}u$ and $h(u)=\mu |\nabla u|\int_{B(0, |x|)} |u|^{q-1}$.\\
We introduce the functional:
$$F(u)(t)=S(t)u_0+\int_0^t S(t-s)(g(u)+h(u))ds.$$
The  proof the existence of a solution $u$ of the Duhamel equation is reduced to the existence of a fixed point of $F$.\\
For the reader's convenience, we recall the following well-known smoothing effect of the heat semigroup:
\begin{equation}\label{Semigroup}
\|S(t)f\|_{L^\infty}\leq \|f\|_{L^\infty}, \quad \|\nabla S(t)f\|_{L^\infty}\leq \frac{C}{\sqrt{t}}\|f\|_{L^\infty}, \forall t>0, \; \forall f\in L^{ \infty}(\R^N).
\end{equation}
Since  we want to prove the existence in $W^{1, \infty}_\beta(\R^N)$, we need more information about the heat semigroup. We give our following result:
\begin{lem}\label{lemm_sg}
For all $0\le m<N$, the heat semigroup satisfies:
\begin{enumerate}
\item[i)] $\|(1+|x|^m)S(t) f\|_{L^\infty}\leq C \|(1+|x|^m) f\|_{L^\infty}$,
\item[ii)] $\|(1+|x|^m)S(t) \nabla f\|_{L^\infty}\leq \frac{C}{\sqrt{t}}\|(1+|x|^m) f\|_{L^\infty}$,
\end{enumerate}
for all $t>0$ and all  $f$ such that $(1+|x|^m) f\in L^\infty$.
\end{lem}
\begin{proof}
We recall that the heat semigroup is defined explicitly by
$$\displaystyle S(t) f(x)=\int_{\R^N}\frac{e^{-\frac{|x-y|^2}{4t}}}{(4\pi t)^{\frac{N}{2}}} f(y) dy.$$
We see that
$$\displaystyle ||x|^m S(t) f(x)|\leq \int_{\R^N}\frac{|x|^m}{|y|^m}\frac{e^{-\frac{|x-y|^2}{4t}}}{(4\pi t)^{\frac{N}{2}}}dy\;\; \|(1+|y|^m) f(y)\|_{L^\infty}.$$
Let $\mathcal{A}=\{y\in \R^N, \; \mbox{such that }\; |x|\leq 2 |y|\}$. We decompose the previous integral as follows:
$$\displaystyle\int_{\R^N}\frac{|x|^m}{|y|^m}\frac{e^{-\frac{|x-y|^2}{4t}}}{(4\pi t)^{\frac{N}{2}}}dy=\int_{\mathcal{A}}\frac{|x|^m}{|y|^m}\frac{e^{-\frac{|x-y|^2}{4t}}}{(4\pi t)^{\frac{N}{2}}}dy+\int_{\R^N\backslash \mathcal{A}}\frac{|x|^m}{|y|^m}\frac{e^{-\frac{|x-y|^2}{4t}}}{(4\pi t)^{\frac{N}{2}}}dy .$$
It is easy to see that
\begin{eqnarray}\label{e_1}
\int_{\mathcal{A}}\frac{|x|^m}{|y|^m}\frac{e^{-\frac{|x-y|^2}{4t}}}{(4\pi t)^{\frac{N}{2}}}dy\leq 2^m \int_{\R^N}\frac{e^{-\frac{|x-y|^2}{4t}}}{(4\pi t)^{\frac{N}{2}}}dy\leq C.
\end{eqnarray}
On the other hand, for  $y\in \R^N\backslash \mathcal{A}  $, we have
$$|x-y|\geq ||x|-|y||\geq \frac{1}{2}|x|.$$
Hence,
$$\displaystyle\int_{\R^N\backslash \mathcal{A}}\frac{|x|^m}{|y|^m}\frac{e^{-\frac{|x-y|^2}{4t}}}{(4\pi t)^{\frac{N}{2}}}dy \leq |x|^m\frac{e^{-\frac{ |x|^2}{16t}}}{(4\pi t)^{\frac{N}{2}}}\int_{\R^N\backslash \mathcal{A}}\frac{dy}{|y|^m}.$$
Since $\R^N\backslash \mathcal{A}=B(0, \frac{|x|}{2})$ and $m<N$, we get
\begin{eqnarray}\label{e_3}
\int_{\R^N\backslash \mathcal{A}}\frac{dy}{|y|^m}=C\int_0^{\frac{|x|}{2}} r^{N-m-1} dr= C(\frac{|x|}{2})^{N-m}.
\end{eqnarray}
Therefore,
$$\displaystyle\int_{\R^N\backslash \mathcal{A}}\frac{|x|^m}{|y|^m}\frac{e^{-\frac{|x-y|^2}{4t}}}{(4\pi t)^{\frac{N}{2}}}dy \leq C (\frac{|x|}{\sqrt{t}})^N\displaystyle e^{-\frac{ |x|^2}{16t}}.$$
Using the fact that $z\mapsto z^N \displaystyle e^{-\frac{z^2}{16}}$ is a bounded function, we obtain
\begin{eqnarray}\label{e_2}
\displaystyle\int_{\R^N\backslash \mathcal{A}}\frac{|x|^m}{|y|^m}\frac{e^{-\frac{|x-y|^2}{4t}}}{(4\pi t)^{\frac{N}{2}}}dy \leq C
\end{eqnarray}
From $(\ref{Semigroup})$, $(\ref{e_1})$ and $(\ref{e_2})$, we deduce the estimate $i)$.\\
Now, we give the proof of $ii)$.
By integration by parts, we get
$$\displaystyle S(t) \nabla f(x)=\int_{\R}\frac{x-y}{2t}\frac{e^{-\frac{(x-y)^2}{4t}}}{(4\pi t)^{\frac{N}{2}}} f(y) dy.$$
Then, we write
$$\displaystyle ||x|^m S(t) \nabla f(x)|\leq \frac{1}{\sqrt{t}}\int_{\R} \frac{|x|^m}{|y|^m}\frac{|x-y|}{2\sqrt{t}}\frac{e^{-\frac{(x-y)^2}{4t}}}{(4\pi t)^{\frac{N}{2}}}dy\;  \|(1+|y|^m) f(y)\|_{L^\infty}.$$
 Since $|z|e^{-|z|^2/4} \le Ce^{-|z|^2/8}$ and the proof of item i) holds if one replaces $4$ by $8$ in the heat kernel,
we see that
$$ \displaystyle\int_{\R^N}\frac{|x|^m}{|y|^m} \frac{|x-y|}{2\sqrt{t}}\frac{e^{-\frac{(x-y)^2}{4t}}}{(4\pi t)^{\frac{N}{2}}} dy \leq 2^m  C\int_{\R^N}\frac{|x|^m}{|y|^m}\frac{e^{-\frac{|x-y|^2}{8t}}}{(4\pi t)^{\frac{N}{2}}}dy \leq C.$$
This concludes the proof of Lemma \ref{lemm_sg}.
\end{proof}

Now, we start the application of a fixed-point argument to solve the Cauchy problem of equation $(\ref{eq_u})$, in the space $ W^{1,\infty}_\beta(\R^N)$, locally in time.\\
Let $T>0$ and consider  $C([0,T], W^{1,\infty}_\beta(\R^N))$  the space of all continuous functions from $[0, T]$ into $W^{1,\infty}_\beta(\R^N)$ equipped with the norm
$$\|u\|_{L^\infty_t(W^{1, \infty}_\beta)}= \sup_{t\in [0, T]}(\|(1+|x|^\beta)u\|_{L^\infty}^2+ \|(1+|x|^\beta)\nabla u\|_{L^\infty}^2)^{\frac{1}{2}}.$$
Our next goal is to find a positive constant $r$ such that the function $F$ is a strict contraction in $B(0, r)$, where $B(0, r)$ is the ball in   $C([0,T], W^{1,\infty}_\beta(\R^N))$ of center $0$ and radius $r$.\\
In a first step, we prove that $F$ is locally lipschitz continuous.\\
 Let $r>0$, for any $u_1, u_2\in B(0, r)$, we write
 $$(F(u_1)-F(u_2))(t)=\int_0^t S(t-s)(g(u_1)-g(u_2)+h(u_1)-h(u_2))ds. $$
 Applying the above Lemma, we obtain
 \begin{eqnarray}\label{C1}
 \|(1+|x|^\beta)(F(u_1)-F(u_2))\|_{L^\infty}\!&\leq& \!C \Big[\int_0^t \|(1+|x|^\beta)(g(u_1)-g(u_2))\|_{L^\infty}\nonumber\\
 &+&\!\int_0^t\|(1+|x|^\beta)(h(u_1)-h(u_2))\|_{L^\infty}\Big],
 \end{eqnarray}
 and
  \begin{eqnarray}\label{C2}
 &&\|(1+|x|^\beta)\nabla (F(u_1)-F(u_2))\|_{L^\infty}\\
 &&\!\leq \int_0^t \frac{C}{\sqrt{t-s}} \|(1+|x|^\beta)(g(u_1)-g(u_2))\|_{L^\infty}+\!\int_0^t \frac{C}{\sqrt{t-s}}\|(1+|x|^\beta)(h(u_1)-h(u_2))\|_{L^\infty}.\nonumber
 \end{eqnarray}
Since  for any $u_1, u_2\in \!B(0, r) $, we have,
 $\|u_i\|_{L^\infty}\!\le\!\!\|(1\!+\!|x|^\beta)u_i\|_{L^\infty}\!\le\!
 r$,
 it is easy to prove that,
 \begin{eqnarray}\label{C3}
 \|(1+|x|^\beta)(g(u_1)-g(u_2))\|_{L^\infty}\leq C r^{p-1}\|(1+|x|^\beta)(u_1-u_2)\|_{L^\infty}.
 \end{eqnarray}
 Now, we estimate $\|(1+|x|^\beta)(h(u_1)-h(u_2))\|_{L^\infty}$. We write
 \begin{eqnarray*}
 |h(u_1)-h(u_2)|
 &\leq& \mu\Big( |\nabla u_1-\nabla u_2|\int_{B(0, |x|)}|u_1|^{q-1}+|\nabla u_2|\int_{B(0, |x|)}\Big||u_1|^{q-1}-|u_2|^{q-1}\Big|\Big)
 \end{eqnarray*}
 On the one hand, since $\|(1+|y|^\beta)u_1\|_{L^\infty}\leq r$ and $(q-1)\beta> N$, we have
 $$\int_{B(0, |x|)}|u_1|^{q-1}\leq \|(1+|y|^\beta)u_1\|_{L^\infty}^{q-1}\int_{\R^N} \frac{dy}{(1+|y|^\beta)^{q-1}}\leq Cr^{q-1}.$$
 On the other hand, we write
 $$\int_{B(0, |x|)}\Big||u_1|^{q-1}-|u_2|^{q-1}\Big|\leq \int_{B(0, |x|)}\frac{ (1+|y|^\beta)^{q-1}\Big||u_1|^{q-1}-|u_2|^{q-1}\Big|}{(1+|y|^\beta)^{q-1}}dy.$$
 Since $q>2$ by $(\ref{hyp})$, we see that
 $$(1+|y|^\beta)^{q-1}\Big||u_1|^{q-1}-|u_2|^{q-1}\Big|\leq C r^{q-2}(1+|y|^\beta)|u_1-u_2|.$$
 Thus,
  $$\displaystyle \int_{B(0, |x|)}\Big||u_1|^{q-1}-|u_2|^{q-1}\Big|\leq Cr^{q-2} \|(1+|x|^\beta)(u_1-u_2)\|_{L^\infty} .$$
 Since $ \|\nabla u_2\|_{L^\infty} \leq \|(1+|x|^\beta)\nabla u_2\|_{L^\infty} \leq r$, we deduce that
  \begin{eqnarray}\label{C4}
 \|(1+|x|^\beta) (h(u_1)-h(u_2))\|_{L^\infty}
 &\leq&  Cr^{q-1}\Big( \|(1+|x|^\beta)(u_1-u_2)\|_{L^\infty} \nonumber\\
 &&+ \|(1+|x|^\beta)\nabla(u_1-u_2)\|_{L^\infty}\Big )
 \end{eqnarray}
 Collecting estimates $(\ref{C1})$, $(\ref{C2})$, $(\ref{C3})$ and
 $(\ref{C4})$, we obtain
 $$\|F(u_1)-F(u_2)\|_{L_t^\infty(W^{1,\infty}_\beta)}\leq C(r^{p-1}+r^{q-1})(T^2+T)^{\frac{1}{2}}\|u_1-u_2\|_{L_t^\infty(W^{1,\infty}_\beta)}.$$
 If we assume $T=T(r)$ small enough such that
 $C(r^{p-1}+r^{q-1})(T^2+T)^{\frac{1}{2}}\leq \frac{1}{2} $, we obtain
 $$\|F(u_1)-F(u_2)\|_{L_t^\infty(W^{1,\infty}_\beta)}\leq \frac{1}{2}\|u_1-u_2\|_{L_t^\infty(W^{1,\infty}_\beta)}.$$
 ~\\
 In the next step, we prove that $F(B(0, r))\subset B(0, r)$. We rewrite, for $u\in B(0,r)$
 $$\|F(u)\|_{L_t^\infty(W^{1,\infty}_\beta)}=\|F(u)-F(0)+ F(0)\|_{L_t^\infty(W^{1,\infty}_\beta)}\leq \frac{r}{2}+\|F(0)\|_{L_t^\infty(W^{1,\infty}_\beta)}.$$
 According to Lemma \ref{lemm_sg}, we have
 $$ \|(1+|x|^\beta)S(t)u_0 \|_{L^\infty}\leq C \|(1+|x|^\beta)u_0\|_{L^\infty},$$
 and
 $$ \|(1+|x|^\beta)S(t)\nabla u_0 \|_{L^\infty}\leq C \|(1+|x|^\beta)\nabla u_0\|_{L^\infty}.$$
 It follows then  that
 $$\|F(0)\|_{L_t^\infty(W^{1,\infty}_\beta)}\leq C\|u_0\|_{L_t^\infty(W^{1,\infty}_\beta)}.$$
 If we assume that
 $$ C\|u_0\|_{L_t^\infty(W^{1,\infty}_\beta)}\leq \frac{r}{4}, $$
 then
 $$\|F(u)\|_{L_t^\infty(W^{1,\infty}_\beta)}\leq \frac{3}{4}r.$$
We  conclude that,  there exist   $r>0$  such that the function $F: B(0, r)\to B(0, r)$ is a strict contraction and   that $F$ admits a unique fixed point $u\in B(0,r)$.\\
\section*{Acknowledgments} The authors would like to
thank the referee for useful comments which helped to improve this paper.

\end{document}